\documentclass[leqno,11pt,letterpaper, english]{amsart}
\usepackage[usenames,dvipsnames]{color}
\usepackage[colorlinks=true,linkcolor=Red,citecolor=Green]{hyperref}
\usepackage{amsmath,amssymb,amsthm,graphicx,url}
\usepackage{epstopdf}
\usepackage{color}
\usepackage{dsfont}
\usepackage[T1]{fontenc}
\newcommand{\xqedhere}[1]{%
    \rlap{%
         \hbox to#1{%
           \hfil
           \llap{%
               \ensuremath{\square}
           }%
       }%
   }%
}

\def\pasdegrille{\let\grille = \pasgrille}

\def\aat#1#2#3{
\divide \dimen1 by 48 \dimen3=\dimen1 \multiply \dimen1 by #1
\advance \dimen1 by -\dimen3 \divide \dimen1 by 101 \multiply
\dimen1 by 100 \divide \dimen2 by \count11 \multiply \dimen2 by #2
\setbox0=\hbox{#3}\ht0=0pt\dp0=0pt
  \rlap{\kern\dimen1 \vbox to0pt{\kern-\dimen2\box0\vss}}\dimen1= \wd1
\dimen2=\ht1}
\def\pasgrille{
\count12= \dimen1 \divide \count12 by 50 \divide \dimen2 by
\count12 \count11 =\dimen2 \ \divide \dimen1 by 48
\setlength{\unitlength}{\dimen1} \smash{\rlap{\ }} \dimen1= \wd1
\dimen2=\ht1 }
\def\grille{
\count12= \dimen1 \divide \count12 by 50 \divide \dimen2 by
\count12 \count11 =\dimen2 \ \divide \dimen1 by 48
\setlength{\unitlength}{\dimen1}
\smash{\rlap{\graphpaper[1](0,0)(50, \count11)}} \dimen1= \wd1
\dimen2=\ht1 }

\pasdegrille

\newcommand{\be}{\begin{equation}}
\newcommand{\ee}{\end{equation}}

\theoremstyle{plain}

\newtheorem{thm}{Theorem}

\newtheorem{prop}{Proposition}[section]
\newtheorem{cor}[prop]{Corollary}
\newtheorem{lem}[prop]{Lemma}

\theoremstyle{definition}

\newtheorem{rem}[prop]{Remark}

\numberwithin{equation}{section}

\def\squarebox#1{\hbox to #1{\hfill\vbox to #1{\vfill}}}

\usepackage{amsxtra}

\ifx\pdfoutput\undefined
  \DeclareGraphicsExtensions{.pstex, .eps}
\else
  \ifx\pdfoutput\relax
    \DeclareGraphicsExtensions{.pstex, .eps}
  \else
    \ifnum\pdfoutput>0
      \DeclareGraphicsExtensions{.pdf}
    \else
      \DeclareGraphicsExtensions{.pstex, .eps}
    \fi
  \fi
\fi

\title[Propagation of smallness and control]{Propagation of smallness\\ and control for heat equations}

\author[N. Burq]{Nicolas Burq}
\address{Universit{\'e} Paris-Saclay, Math{\'e}matiques, UMR 8628 du CNRS, B{\^a}t 307, 91405  Orsay Cedex, France,   and Institut Universitaire de France}
\email{Nicolas.burq@universite-paris-saclay.fr}
\author[I. Moyano]{Iv\'an Moyano}
\address{Universit\'e de Nice Sophia-Antipolis
Parc Valrose, Laboratoire J.A. Dieudonn\'e,
UMR 7351 du CNRS
06108 NICE Cedex 02
FRANCE}
\email{Ivan.Moyano@unice.fr}

\usepackage{amssymb}
\usepackage{amsmath, amsthm, amsopn, amsfonts}

\def\11{{\rm 1~\hspace{-1.4ex}l} }
\def\R{\mathbb R}

\def\N{\mathbb N}

\begin{document}

\begin{abstract}
In this note we investigate propagation of smallness properties for solutions to heat equations. We consider spectral projector estimates for the Laplace operator with Dirichlet or Neumann boundary conditions on a Riemanian manifold with or without boundary. We show that using the new approach for the propagation of smallness from Logunov-Malinnikova~\cite{Lo18, Lo18-1, LoMa17}  allows to extend the spectral projector type estimates from Jerison-Lebeau~\cite{JeLe99} from localisation on open set to localisation on arbitrary sets of non zero Lebesgue measure; we can actually go beyond and consider sets of non vanishing $d- \delta$ ($\delta >0$ small enough) Hausdorff measure. We show that these new spectral projector estimates  allow to extend the Logunov-Malinnikova's propagation of smallness results to solutions to heat equations. Finally we apply these results to the null controllability of heat equations with controls localised on sets of positive Lebesgue measure. A main novelty here with respect to previous results is that we can drop the constant coefficient assumptions (see~\cite{AEWZ14, ApEs13}) of the Laplace operator (or analyticity assumption, see~\cite{EsMoZh17,LeMo19}) and deal with Lipschitz coefficients. Another important novelty is that we get the first (non one dimensional) exact controllability results with controls supported on zero-measure sets. \\
 \begin{center} \rule{0.6\textwidth}{.4pt} \end{center}\end{abstract}   

\ \vskip -1cm \noindent\hfil\rule{0.9\textwidth}{.4pt}\hfil \vskip 1cm 
 \maketitle   
 \section{Introduction}  
We are interested in this note in understanding the propagation of smallness and control for solutions to heat equations and their connections with the propagation of smallness for  high frequency sum of eigenfunctions of the Laplace operator on a  { compact} Riemanian manifold $(M,g)$ with boundary. Let $\Delta$  be the  Laplace-Beltrami operator on $M$\footnote{We will simply denote by $\Delta$, without emphasising the dependence on the metric $g = g(x)$, the variable coefficient operator $\Delta = \frac{1}{\sqrt{\det g}} \partial_i \left( \sqrt{\det g} g^{ij} \partial_j \right)$ or the more general operator defined in \eqref{delta}.} and let $(e_k)$ be a family of eigenfunctions of $- \Delta$, with eigenvalues $\lambda_k^2 \rightarrow + \infty$ forming a Hilbert basis of $L^2(M)$. 
$$ -\Delta e_k = \lambda_k ^2 e_k, \qquad e_k \mid_{\partial M} =0 \text{ (Dirichlet condition) or } \partial_\nu e_k \mid_{\partial M} =0 \text{ (Neumann condition)}.$$
Now, we consider any arbitary finite linear combination of the form
$$\phi= \sum_{\lambda_k \leq \Lambda} u_k e_k (x), $$
and given a small subset $E\subset M$ (of positive Lebesgue measure or at least not too small in a sense to be made precise later), 
we want to understand how $L^p$ norms of the restrictions of $\phi$ on the set $E$ dominate Sobolev norms of $\phi$ on $M$. 

In the case of domains and constant coefficient Laplace operator and subsets of positive Lebesgue measures, or in the case of Lipschitz metrics and open subsets $E$, this is now quite well understood~\cite{AEWZ14, JeLe99}. Here we shall be interested in the two cases where  $M$ is a $W^{2, \infty}$ compact manifold of dimension $d$ with or without boundary (endowed with a Lipschitz metric) and observation domains $E$ of positive Lebesgue measure or even of positive $(d- \delta)$- dimensional Hausdorff content for $\delta >0$ small enough, but depending only on the dimension of the manifold $M$.

Here and below by $W^{2, \infty}$  manifolds, we mean that the change of charts are $C^1 \cap W^{2, \infty}$ maps ($C^1$ with  second order distribution derivatives bounded a.e. or equivalently the derivatives of the change of charts are Lipschitz functions).
We allow slightly more general operators than Laplace-Beltrami operators and assume that $M$ is endowed with a Lipschitz (definite positive) metric $g$ and a Lipschitz (positive) density $\kappa$. Let 
\begin{equation}\label{delta}
 \Delta = \frac 1 {\kappa(x)} \text{div} g^{-1}(x) \kappa(x) \nabla_x= \text{div}_\kappa \nabla_g
 \end{equation} be the corresponding Laplace operator. When $\kappa(x) = \sqrt{\det g(x)}$, we recover the usual Laplace-Beltrami operator on $(M,g)$.

In all the results below, the manifold $M$ will be assumed to satisfy the $W^{2, \infty}$ regularity above and unless stated explicitely otherwise, $\Delta$ stands for the operator defined by~\eqref{delta} with Dirichlet or Neumann boundary condition if $\partial M \neq \emptyset $.
Recall that the $d$-Hausdorff content (or measure) of a set $E\subset \mathbb{R}^n$ is 
$$ \mathcal{C}_{\mathcal{H}}^d (E) = \inf \{ \sum_j r_j ^d; E \subset \cup_jB(x_j, r_j)\},$$
and the Hausdorff dimension of $E$ is defined as 
$$ \text{dim}_\mathcal{H} (E) = \inf \{d;   \mathcal{C}_{\mathcal{H}}^d(E) =0 \}.$$
We shall denote by $|E|$ the Lebesgue measure of the set $E$. Let us recall that the Hausdorf content of order $n$ is equivalent to the Lebesgue measure, 
$$ \exists C_d, c_d>0; \forall A \text{ borelian set}, c_d |A| \leq \mathcal{C}^d (A) \leq C_d |A|,$$
and  
\begin{equation}\label{equivalence}
  \mathcal{C}_{\mathcal{H}}^d (E)>0 \Rightarrow  \forall d'<d, \mathcal{C}_{\mathcal{H}}^{d'} (E)\geq \inf( 1,  \mathcal{C}_{\mathcal{H}}^d (E))
  \end{equation}
  
 (indeed, $\sum_j r_j ^{d'} \geq 1$ if there exists $i_0$ such that $r_{i_0} \geq 1$, and otherwise, $\sum_j r_j ^{d'} \geq \sum_j r_j ^{d}$).
 
The value of the Hausdorff content is not invariant by diffeomorphisms, but the Hausdorff dimension is invariant by Lipschitz diffeormorphisms, as shown by 
\begin{prop}\label{invar}
Let $\phi: \mathbb{R}^n \rightarrow \mathbb{R}^n$ be a Lipschitz diffeomorphism, such that 
\begin{equation}
\label{lips}
 \| \nabla_x (\phi) \|_{L^\infty} \leq C.
 \end{equation}
Then, for any $\sigma > 0$,
\begin{equation}
 \mathcal{C}_{\mathcal{H}}^{\sigma} (E) >m \Rightarrow 
 \mathcal{C}_{\mathcal{H}}^{\sigma} (\phi(E)) >C^{-\sigma}m.
 \end{equation}
 \end{prop}

\begin{proof}

 Indeed, assume that 
 $ E \subset \cup_j B(x_j, r_j)$. Then $\phi(E) \subset  \cup_j \phi(B(x_j, r_j))$. 
 But, according to~\eqref{lips}, with $y_j = \phi (x_j)$, we have 
  $$ \| \phi(x) - \phi(y_j) \| \leq C \| x- x_j\| \Rightarrow \phi(B(x_j, r_j)) \subset B(\phi(x_j), Cr_j), $$
  As a consequence, 
 \begin{multline*}
  \mathcal{C}_{\mathcal{H}}^\sigma (E) = \inf \{ \sum_j r_j^\sigma; E \subset \cup_jB(x_j, r_j)\}\\
  \geq C^{-\sigma }  \inf \{ \sum_j r_j ^\sigma; \phi (E) \subset \cup_jB(y_j, Cr_j)\}
  = C^{-\sigma } \mathcal{C}_{\mathcal{H}}^\sigma (\phi (E)).
  \end{multline*}

\end{proof}

Our first result is the following generalisation of Jerison-Lebeau's work~\cite{JeLe99}.
\begin{thm}\label{spectral}
There exists $\delta\in (0,1)$ (depending only on the dimension of the manifold $M$) such that for any $m>0$, there exists $C, D>0$ such that for any $E_1 \subset M$ with $ |E_1| \geq m$,  any $E_2\subset M$ satisfying 
\begin{equation}\label{cont-haus}
 \mathcal{C}_{\mathcal{H}}^{d- \delta} (E_2) >m,
 \end{equation}
and for any $\Lambda >0$,   
we have
\begin{equation}\label{BORNE}
\phi= \sum_{\lambda_k \leq \Lambda} u_k e_k (x) \Rightarrow \| \phi\|_{L^\infty(M)} \leq C e^{D \Lambda} \| \phi 1_{E_1}\|_{L^1(M)},
\end{equation}
\begin{equation}\label{BORNERAF}
\phi= \sum_{\lambda_k \leq \Lambda} u_k e_k (x) \Rightarrow \| \phi\|_{L^\infty(M)}  \leq C e^{D \Lambda} \sup_{x \in E_2} |\phi (x)|.
\end{equation}

\end{thm}
\begin{rem}
The assumption~\eqref{cont-haus} is {\em not} invariant by change of variables. It has to be understood in a fixed local chart (and we shall prove Theorem~\ref{spectral} in a chart). Taking $0<\delta' <\delta$, we could have replaced it by $\text{dim}_{\mathcal{H}}(E) > d - \delta'$ (which implies $\mathcal{C}^{d- \delta}_{\mathcal{H}} > 0$ and is invariant by Lipshitz diffeomorphisms). Of course replacing $\delta >0$ by any $0<\delta'<\delta$ does not change substantially the final result (as we have no control on the actual value of the constant $\delta$).  For the sake of consistency with~\cite{Lo18, Lo18-1, LoMa17} we kept~\eqref{cont-haus}. 
\end{rem}

\begin{rem} Notice that in Theorem~\ref{spectral} no assumption is made on the set $E_2$ other than the positivity of the Hausdoff content. This implies that in the presence of a boundary, the estimate (\ref{BORNERAF}) also holds when $E_2$ is concentrated arbitrarily close to $\partial M$ (with uniform constants).
\end{rem}
As a consequence of these spectral projector estimates we deduce the following observability estimates and controllability results for the heat equation.

\begin{thm} [Null controllability from sets of positive measure]\label{control}
Let $F \subset (0,T) \times M$ of positive Lebesgue measure. 
Then, there exists $C>0$ such that for any $u_0 \in L^2(M)$ the solution 
$u = e^{t \Delta} u_0$ to the heat equation 
$$\partial _t u - \Delta u =0,     u \mid_{\partial M} =0 \text{ (Dirichlet condition) or } \partial_\nu u \mid_{\partial M} =0 \text{ (Neumann condition)}, u \mid_{t=0} = u_0,
$$ satisfies (recall that $\kappa$ is defined in~\eqref{delta}):
\begin{equation}\label{obs} 
\| e^{T \Delta } u_0 \|_{L^2(M)} \leq C \int_F | u|(t,x) \kappa(x)dx dt.
\end{equation}
As a consequence, for all $u_0, v_0 \in L^2(M)$ there exists $f\in L^\infty(F)$ such that the solution to 
\begin{equation}
\begin{gathered}
(\partial_t - \Delta ) u = f1_{F} (t,x), \quad u \mid_{t=0} = u_0,\\
u \mid_{\partial M} =0 \text{ (Dirichlet condition) or } \partial_\nu u \mid_{\partial M} =0 \text{ (Neumann condition)},
\end{gathered}
\end{equation}
satisfies 
$$ u \mid_{t=T} =e^{T\Delta} v_0.$$
\end{thm}
\begin{thm}[Observability and exact controllability from zero measure sets]\label{control-haus} 
There exists $\delta\in (0,1)$ (depending only on the dimension of the manifold $M$) which depends only on the dimension of the manifold $M$, such that for any  $E \subset M$ of positive  ($d- \delta$) dimensional Hausdorff content, and any $J\subset (0,T) $ of positive Lebesgue measure,  there exists $C>0$ such that  for any $u_0 \in L^2(M)$ the solution 
$u = e^{t \Delta} u_0$ to the heat equation 
$$\partial _t u - \Delta u =0,     u \mid_{\partial M} =0 \text{ (Dirichlet condition) or } \partial_\nu u \mid_{\partial M} =0 \text{ (Neumann condition)},
$$ satisfies 
\begin{equation}\label{obsbis} 
\| e^{T \Delta } u_0 \|_{L^2(M)} \leq C \int_{J} \sup_{x\in E} | u|(t,x)  dt.
\end{equation}
As a consequence, under the additional assumption that $E$ is a closed subset of $M$, for all $u_0 , v_0 \in L^2(M)$ there exists $\mu $ a Borel measure supported on $(0,T) \times E$ such that  the solution to 
\begin{equation}\label{eq-mesure thm}
\begin{gathered}
(\partial_t - \Delta ) u = \mu (t,x)1_{J \times E}, \quad u \mid_{t=0} = u_0,\\
u \mid_{\partial M} =0 \text{ (Dirichlet condition) or } \partial_\nu u \mid_{\partial M} =0 \text{ (Neumann condition)},
\end{gathered}
\end{equation}
satisfies 
$$ u \mid_{t\geq T} =e^{t\Delta} v_0.$$
\end{thm}
We refer to Section~\ref{sec.5} (see \eqref{eq-mesure}) for the precise meaning of \eqref{eq-mesure thm}.
Actually, we can even go a step further and show that the $d+1$ dimensional heat equation can be steered to zero by using measure-valued controls supported on a set of space-time Hausdorff measure $d - \delta$. 
\begin{thm}[Observability and exact controllability using controls localised at fixed times]\label{control-hausbis} 
Take $\delta\in (0,1)$ as in Theorem~\ref{spectral}. Let  $m>0$, $\tau \in (0,1)$ and $D>0$. There exists  $C>0$,  such that if  $E_1\subset M$ satisfies $|E_1| \geq m$ or $E_2 \subset M$ satisfies $\mathcal{C}_{\mathcal{H}}^{d- \delta} (E_2)\geq m$. 
then for any sequence $(s_n)_{n\in \mathbb{N}}$,
$$0<\cdots <  s_n<  \cdots  < s_0< T $$ converging not too fast to $0$, i.e.,
$$  \exists \tau \in (0,1); \forall n \in \N,  (s_{n} - s_{n+1}) \geq \tau ( s_{n-1} - s_{n}), $$
we have that  for any $u_0 \in L^2(M)$, the solution 
$u = e^{t \Delta} u_0$ to the heat equation 
$$\partial _t u - \Delta u =0,     u \mid_{\partial M} =0 \text{ (Dirichlet condition) or } \partial_\nu u \mid_{\partial M} =0 \text{ (Neumann condition)},
$$ satisfies 
\begin{equation}\label{obster} 
\| e^{T \Delta } u_0 \|_{L^2(M)} \leq   C \sup_{n \in \N} e^{-\frac D{s_n- s_{n+1}} }\int_{E_1} |e^{{s_n\Delta} } u_0| (s_n,x) |dx \end{equation}
and \begin{equation}\label{obsterbis} 
\| e^{T \Delta } u_0 \|_{L^2(M)} \leq  C \sup_{n \in \N, x\in E_2} e^{-\frac D{s_n- s_{n+1}} }| e^{s_n\Delta}  u_0 |(s_n, x).
\end{equation}
As a consequence, under the additional assumption that $E_1$ is a closed subset of $M$, given any sequence $(t_n)_{n\in \mathbb{N}}$,
$$0< t_0 < \cdots< t_n < \cdots< T $$ converging not too fast to $T$,
\begin{equation}\label{nottoofast} \exists 0<\tau<1; \forall n \in \N,  (t_{n+1} - t_{n} )\geq \tau ( t_{n} - t_{n-1}),
\end{equation} for all $u_0, v_0 \in L^2(M)$ 
there exists $(f_j) $ a sequence of functions on  $E_1$ such that 
$$ \sum_j e^{\frac D{(t_{j+1}-t_j)} }\|f_j\|_{L^\infty(E_1)}  < +\infty, $$ and the solution to 
\begin{equation}\label{solmesurebis}
\begin{gathered}
(\partial_t - \Delta ) u = \sum_{j=1}^{+\infty} \delta_{t=t_j}\otimes f_j (x)1_{E_1},  \quad u \mid_{t=0} = u_0,\\
u \mid_{\partial M} =0 \text{ (Dirichlet condition) or } \partial_\nu u \mid_{\partial M} =0 \text{ (Neumann condition)},
\end{gathered}
\end{equation}
satisfies 
$$ u \mid_{t> T} =e^{t\Delta} v_0.$$

Similarly  under the additional assumption that $E_2$ is a closed subset of $M$, given any sequence $(t_n)_{n\in \mathbb{N}}$,
$$J= \{ 0\leq t_0 < \cdots< t_n < \cdots< T\} $$ converging not too fast to $T$ as in~\eqref{nottoofast},
 for all $u_0, v_0 \in L^2(M)$,
there exists $(\mu_j) $ a sequence of Borel measure supported on $E_2$ such that 
$$ \sum_j e^{\frac D{(t_{j+1}-t_j)} }|\mu_j| (E_2) < +\infty, $$ and the solution to 
\begin{equation}\label{solmesure}
\begin{gathered}
(\partial_t - \Delta ) u = \sum_{j=1}^{+\infty} \delta_{t=t_j} \otimes \mu_j (x)1_{E_2},  \quad u \mid_{t=0} = u_0,\\
u \mid_{\partial M} =0 \text{ (Dirichlet condition) or } \partial_\nu u \mid_{\partial M} =0 \text{ (Neumann condition)},
\end{gathered}
\end{equation}
satisfies 
$$ u \mid_{t> T} =e^{t\Delta} v_0.$$
\end{thm}
The meaning of solving~\eqref{solmesurebis},~\eqref{solmesure} is also explained in Section~\ref{sec.5}. 
\begin{rem}
We have $t_{j+1} - t_j\leq T-t_j$. As a consequence we get that 
$$\|f_j\|_{L^\infty(E_1)} \leq C e^{-\frac{D} {T-t_j}}, \qquad |\mu_j|(E_1) \leq C e^{-\frac{D} {T-t_j}}$$
which means that our controls are exponentially small when $j\rightarrow + \infty$ ($t\rightarrow T$).
\end{rem}
The plan of the paper is as follows. In Section~\ref{sec.2} we show how, for manifolds without boundaries, the estimates for spectral projectors, (Theorem~\ref{spectral}) follows quite easily from Logunov-Malinnikova's results~\cite{Lo18, Lo18-1, LoMa17} combined with Jerison-Lebeau's method~\cite{JeLe99}. Then in Section~\ref{sec.3}, we show how to extend the results to the case of manifolds with boundaries. When the manifold is smooth, this is quite standard as we can extend it by reflexion around the boundary using geodesic coordinate. This allows to define a new $W^{2, \infty}$ manifold without boundary (the double manifold), which is topologically two copies of the original manifold glued at the boundary, and into which these two copies embed isometrically. At our low regularity level, the use of geodesic coordinate systems is prohibited  and a careful work is required to perform this extension. We actually provide with the natural alternative for geodesic systems (see Proposition~\ref{geodesic}). We believe that this construction of the double manifold at this low regularity level has an interest of its own.  In Section~\ref{sec.4} we prove the propagation of smallness and observation estimates for solutions to heat equation (estimates~\eqref{obs}, \eqref{obsbis} and~\eqref{obster} in Theorems~\ref{control}, \ref{control-haus} and~\ref{control-hausbis}), by adapting a proof in Apraiz et al.~\cite{AEWZ14}, which in turn relied on a mixing of ideas from Miller~\cite{mi10} and Phung-Wang~\cite{PhWa13}, following the pionneering work by Lebeau-Robbiano~\cite{LeRo95-1}. Finally, in Section~\ref{sec.5} we prove the exact controllability results by adapting quite classical duality methods to our setting. Here we also improve on previous results by allowing control supported on a sequence of times (hence measure zero set in time).
\subsection*{Acknowledgements}
	N.B.  is supported by Institut Universitaire de France and ANR  grant ISDEEC, ANR-16-CE40-0013. Part of this work was done while N.B. was in residence during the fall 2019 at the Mathematical Science Research Institute (MSRI)
with support by the NSF grant DMS-1440140. I. M. was partially supported by the European Research Council (ERC) MAFRAN grant under the European's Union Horizon 2020 research and innovation programme (grant agreement No 726386) while he was a research associate in PDEs at DPMMS, University of Cambridge,~U.K.
We would like to thank a referee for pointing to us a mistake in the previous proof of the construction of the double manifold (Section~\ref{sec.3}).
\tableofcontents

\section{Proof of the spectral inequalities for compact manifolds}\label{sec.2}
In this section we give a proof of Theorem~\ref{spectral} in the case of a manifold without boundary. We first show in Section \ref{sec:raf} that the estimate~\eqref{BORNE} is actually a straightforward consequence of the results obtained by Logunov and Malinnikova~\cite{LoMa17}. In Section \ref{sec.2.1} we combine \cite{LoMa17} with the spectral estimates on open sets obtained by Jerison and Lebeau (cf. \cite{JeLe99}) to get (\ref{BORNERAF}) when $\partial M = \emptyset$. \par 

We deal with the case $\partial M \not= \emptyset$ in Section \ref{sec.3}.

\subsection{The spectral inequality for very small sets implies the spectral inequality for non zero measure sets}
\label{sec:raf}

Here we prove that ~\eqref{BORNERAF} implies~\eqref{BORNE}. 
Assume that $|E_1|>m$ is given. 

 and consider 
$$\phi= \sum_{\lambda_k \leq \Lambda} u_k e_k (x), \qquad \textrm{with} \qquad \| \phi \|_{L^2(M)} =1.$$ 
Let $F\subset M$ with $ |F| > \frac m 2$. According to~\eqref{equivalence} (and the fact that the measure and $\mathcal{C}^d$ are equivalent), we have 
$$ \mathcal{C}_{\mathcal{H}}^d(F)> c_d \frac m 2 \Rightarrow \mathcal{C}_\mathcal{H}^{d- \delta } (F) \geq \min(1, c_d \frac m 2 ). $$
Now, according to~\eqref{BORNERAF}, we have for any such $F\subset M, |F| \geq \frac m 2$, 
\begin{equation}
\label{estim}\| \phi\|_{L^\infty(M)}  \leq C e^{D \Lambda} \sup_F |\phi (x)|. 
\end{equation}
with constants uniform as long as $|F| \geq \frac m 2$.
Let 
\begin{equation*}
F = \Bigl\{ x\in {{E_1}}; \quad |\phi(x)| \leq \frac{1} {2 C} e^{- D \Lambda} \| \phi\|_{L^\infty(M)}\Bigr\}.
\end{equation*} 
If $|F|\geq\frac m 2$, we have  
\begin{equation*}
\| \phi\|_{L^\infty(M)} \leq  C e^{D \Lambda} \sup_F |\phi (x)| \leq \frac {\| \phi\|_{L^\infty(M)}} 2,
\end{equation*} which shows that $F$ cannot satisfy~\eqref{estim} (because $\| \phi\|_{L^2(M)} = 1 \Rightarrow \phi \not \equiv 0$). Hence, $|F| <\frac m 2$ and consequently, 
$$ \int_{{E_1}} |\phi(x)| dx  \geq \int_{{E_1\setminus F}} |\phi(x)| dx  \geq \frac{ |{{E_1}}|  } { (4C) } e^{-D \Lambda } \| \phi\|_{L^\infty(M)}
$$ which implies ~\eqref{BORNE}.

\subsection{Proof of the precised estimate for compact manifolds}
\label{sec.2.1}

Let $m> 0$ and let $E_2 \subset M$ be a given set with $\mathcal{C}_{\mathcal{H}}^{d- \delta} (E_2) >m$. Our goal is to obtain estimate \eqref{BORNERAF} in this case. \par

We first localize the estimate on a coordinate patch. Since $M$ is compact, there exists a finite covering 
$$ M \subset \bigcup_{j=1}^N U_j ,$$ 
and $W^{2,\infty}$ diffeomorphisms  $\psi_j: V_j \rightarrow \mathbb{R}^d$, with $V_j$ a neighborhood of $U_j$ in $M$. From~\cite[Theorem 14.6]{JeLe99} there exists $C,D>0$ depending only on $M$ such that for any $j\in \{1, \dots, N\}$, 
\begin{equation}\label{jerison-lebeau}
 \| \phi \|_{L^2(M)}\leq C e^{D\Lambda} \| \phi\|_{L^2(U_j)}.
 \end{equation}

Let $j_0$ such that 
$$\mathcal{C}_{\mathcal{H}}^{d- \delta} (E_2 \cap U_{j_0}) > \frac{m}{N}.$$
We now work in the coordinate patch, $U_{j_0}$ and define the sets 
\begin{equation*}
V= \psi_{j_0}(V_{j_0}), \qquad U= \psi_{j_0}(U_{j_0}), \qquad F = \psi_{j_0}(E_2 \cap U_{j_0}).
\end{equation*} Observe that, as $\psi_{j_0}$ is a diffeomorphism of class $W^{2,\infty}$ by hypothesis, we must have
\begin{equation}
\mathcal{C}_{\mathcal{H}}^{d - \delta} (F) > C(\psi_0) \frac{m}{N}.
\label{eq:content Hausdorff preserve par diffeo}
\end{equation} Now,  denote by $f_k$ and $\varphi$ the images of $e_k$ and $\phi$ by the push forward $(\psi_{j_0})_*$, which are defined on $V$. Consider the functions
\begin{equation*}
 u(t,x) = \sum_{\lambda_k \leq \Lambda} u_k \frac{ \sinh(\lambda_k t )  }{\lambda_k} f_k(x), \qquad \varphi(x) := \sum_{\lambda_k \leq \Lambda} u_k f_k(x) = \partial_t u \mid_{t=0},
\end{equation*} for $(u_k)_{k}$ given. 
Here by convention we set, for $\lambda=0$, 
${\frac{ \sinh(\lambda t )  }{\lambda}}= t $.
 We have  
$$ \Bigl( \frac 1 {\kappa} \text{div} g^{-1} \kappa \nabla_x + \partial_t^2\Bigr) u =0 \Leftrightarrow  \Bigl(\text{div} g^{-1} \kappa \nabla_x + \partial_t\kappa \partial_t \Bigr) u=0. $$ 
 Consider, for $T_2 > T_1 > 0$ the sets 
\begin{equation*}
\mathcal{K}  := [-T_1, T_1] \times \overline{U}, \qquad \Omega  := (-T_2, T_2) \times V,   \qquad  E= \{0\} \times F
\end{equation*} which by construction satisfy the inclusions $E \subset \mathcal{K}\subset \Omega$. Next, thanks to \eqref{eq:content Hausdorff preserve par diffeo}, we can write
\begin{equation}\label{2.31}
 \mathcal{C}_{\mathcal{H}}^{n-1- \delta} (E) > m', \qquad \textrm{for} \quad n= d+1, \quad m' = C(\psi_0) \frac{m}{N}.
 \end{equation}
For sufficiently small $\delta >0$ we can now apply~\cite[Theorem 5.1]{LoMa17} and get {{from~\eqref{2.31}}}
\begin{equation}\label{log-mal} 
\sup_{\mathcal{K}} |\nabla_{t,x} u| \leq C \bigl( \sup_E |\nabla_{t,x} u| \bigr)^{\alpha} \bigl( \sup_{\Omega} |\nabla_{t,x} u|\Bigr)^{1- \alpha}.
\end{equation}
We now need  a variant of  Sobolev embeddings, which we prove for the reader's convenience: 
\begin{prop}\label{sobolev}There exists $\sigma>0$ such that with 
$$ \mathcal{H}^\sigma= D( (- \Delta)^{\sigma/2})$$ endowed with its natural norm 
$$\| u\|_{\mathcal{H}^\sigma}= \Bigl(\sum_{k} |u_k |^2 (1+ \lambda_k)^{2\sigma} \Bigr)^{1/2} ,
$$
we have
$$\| \nabla_x u \|_{L^\infty}+  \| u \|_{L^\infty} \leq C \| u \|_{\mathcal{H}^\sigma}.$$
\end{prop}
{
\begin{rem}
For smooth metrics and compact manifolds (without boundary), the space $\mathcal{H}^\sigma$ coincide with the usual Sobolev spaces $H^\sigma$, and Lemma~\ref{sobolev} is just the usual Sobolev injection. At our level of regularity it is no more case, as the spaces $H^\sigma$ and $\mathcal{H}^\sigma$ coincide for $0\leq \sigma \leq 2$ but no further.
\end{rem}}
\begin{proof}
We start with a Lemma about eigenfunctions of the Laplace operator
\begin{lem}\label{fonctionspropres} For all $\sigma_1 > \frac d 2$ there exists $C>0$ such that for all eigenfunctions of the Laplace operator we have 
$$ \| \nabla_x e_k \|_{L^\infty} \leq C (1+ \lambda_k ) ^{1+ \sigma_1}  \| e_k\|_{L^2}, \qquad \|  e_k \|_{L^\infty} \leq C ( \lambda_k ) ^{1+ \sigma_1}  \| e_k\|_{L^2}.$$ 
\end{lem}
\begin{proof}[Proof of Lemma~\ref{fonctionspropres}]
We start with the bound for $\| e_k\|_{L^\infty}$.  For $n \leq 3$, it follows from elliptic regularity 
$$ \| e_k \|_{H^2} \leq C ( \| \Delta e_k\|_{L^2}+ \| e_k \|_{L^2}) \leq C ( 1 + \lambda^2),$$ which implies for $0 \leq s \leq 2$,
$$ \| e_k\|_{H^s} \leq C ( 1 + \lambda^2)^{s/2}, $$
and Lemma~\ref {fonctionspropres} follows from Sobolev embeddings.
For higher dimensions, we shall use the following results from~\cite{HaLi97} about weak solutions to 
 \begin{equation}\label{weak}
  -\sum_{i,j} \partial_{y_i} a_{i,j} \partial_{y_j} w + c w =f.\end{equation}
  with 
  $$ \lambda |\xi|^2 \leq \sum_{i,j =1}^n a_{i,j}(x) \xi_i \xi_j \leq \Lambda |\xi|^2$$
   
 \begin{thm}[\protect{\cite[Theorem 3.8, combined with Corollary 3.2]{HaLi97} }]Consider $w \in H^1( B(0,1)$ a weak solution to~\eqref{weak}. Assume that $a_{i,j}\in C^0( \overline{B(0,1)}$, $c\in L^n(B(0,1))$,  $f \in L^q(B(0,1))$, for some $q \in (\frac n 2 , n)$. Then $w \in C^\alpha (B(0,1))$, $\alpha = 2 - \frac n q \in (0,1)$. Moreover  there exists $M = M( \lambda, \Lambda, \| c\| _{L^n}, \tau)>0$ such that  we have the estimate
  \begin{equation}
 \|w\|_{C^{0,\alpha}(B(0, \frac 1 2))} = \sup_{B(0, \frac 1 2)} |w| +  \sup_{x,y \in B(0, \frac 1 2, \, x \not = y} \frac{|w(x) - w(y)|}{ |x - y|^{\alpha} }     \leq M (\| f\|_{L^q(B(0,1))}  +  \|w \|_{H^1(B(0,1))}), 
  \end{equation}  
 where $\tau$ is the uniform continuity modulus of the functions $a_{i,j}$, 
 $$ \forall y, y' \in B(0,1),  |a_{i,j} (y) - a_{i,j}(y')| \leq \tau (|y-y'|).$$
  \end{thm}
  Using a partition of unity and applying this lemma in charts with $c= - \lambda_k^2 \kappa (x)$, $f=0$, we get that $e_k \in L^\infty$ (with an implicit bound in terms of $\lambda_k$). Then applying again the result with $c=0$, $f(x)= \lambda_k^2 \kappa (x) e_k (x)$, we get (choosing $q = \frac n 2 +0$, i.e. arbitrarily close to $\frac n 2$)
  $$ \| e_k\| _{L^\infty} \leq M ( \lambda_k^2 \| e_k\|_{L^q} + \| e_k \|_{H^1} ) \leq M (  \lambda^2 \| e_k\|_{L^\infty}^{1- \theta} \| e_k \| _{L^2}^\theta + \lambda_k \| e_k \|_{L^2},$$
  with $ \theta = \frac 2 q = \frac 4 n - 0$,
  and consequently, we get 
  $$ \| e_k\|_{L^\infty} \leq M' (\lambda_k ^{\frac 2 \theta } \| e_k \|_{L^2}+ \lambda_k \| e_k \|_{L^2}) \leq C (1+\lambda_k) ^{\frac n 2 + 0} \| e_k \|_{L^2}.$$
  Now, we turn to the estimates for $\|\nabla_x e_k \|_{L^\infty}$. We shall use in this case 
 \begin{thm}[\protect{\cite[Theorem 3.13, combined with Theorem 1.3]{HaLi97} }]Consider $w \in H^1( B(0,1))$ a weak solution to~\eqref{weak}. Assume that $a_{i,j}\in C^\alpha( \overline{B(0,1))}$,  $f \in L^q(B(0,1))$, for some $q >n$. Then $\nabla_y w \in C^\alpha (B(0,1))$, $\alpha = 1 - \frac n q \in (0,1)$. Moreover  there exists $M = M( n, \lambda, \| a_{i,j} \|_{C^\alpha})>0$ such that  we have the estimate
  \begin{multline}
 \|\nabla_y w\|_{C^{0,\alpha}(B(0, \frac 1 2))} = \sup_{B(0, \frac 1 2)} |\nabla_y w| +  \sup_{\smash{y, y' \in B(0, \frac 1 2), \, y' \not = y}} \frac{|\nabla_y w(y) - \nabla_y w(y')|}{ |y - y'|^{\alpha} }   \\
   \leq M (\| f\|_{L^q(B(0,1))}  +  \|w \|_{H^1(B(0,1))}).
 \label{eq:gradient estimate Holder}
 \end{multline}  
   \end{thm} 
    Using again a partition of unity and applying now this lemma in charts with $c=0$, $f(x)= \lambda_k^2 \kappa (x) e_k (x)$, we get (choosing $q = n+0$, i.e. arbitrarily close to $n$)
  $$ \| \nabla_xe_k\| _{L^\infty} \leq M ( \lambda_k^2 \| e_k\|_{L^q} + \| e_k \|_{H^1} ) \leq M (  \lambda^2 \| e_k\|_{L^\infty}^{1- \theta} \| e_k \| _{L^2}^\theta + \lambda_k \| e_k \|_{L^2},$$
  with $ \theta = \frac 2 q = \frac 2 n - 0$,
  and consequently, we get 
  $$ \| e_k\|_{L^\infty} \leq M (\lambda_k ^{2+( \frac n 2 + 0) (1- \theta)  } \| e_k \|_{L^2}+ \lambda_k \| e_k \|_{L^2}) \leq C ( 1+ \lambda_k )^{\frac n 2 +1 + 0} \| e_k \|_{L^2}.$$
\end{proof}
Let us now come back to the proof of Proposition~\ref{sobolev}. From Weyl formula 
$$\lambda_k \sim k^{\frac 1 d}. $$
As a consequence, we have 
\begin{multline}
\| \sum_k u_k  e_k \|_{L^\infty} \leq C \sum_k |u_k| (1+\lambda_k)^{\sigma_1} \\
\leq C \Bigl( \sum_k |u_k |^2 (1+ \lambda_k )^{2\sigma_1+2p}\Bigr)^{1/2} \Bigl( \sum_k (1+ \lambda_k)^ {-2p} \Bigr)^{1/2}\leq C \| u\|_{\mathcal{H}^{\sigma_1+ p}},
\end{multline} as soon as $2p>d$.
\end{proof}
{ We now come back to the proof of~\eqref{BORNERAF}.
Using Sobolev embedding, we observe 
\begin{equation*}
\sup_{\Omega} |\nabla_{t,x} u | \leq C \| \nabla_{t,x} u(t)\|_{\mathcal{H}^{\sigma }}   \leq Ce^{(T_2+1) \Lambda} \|\phi \|_{L^2(M)}. 
\end{equation*} }By definition of $u$, we have 
\begin{equation*}
\nabla_x u \mid_{E}  = 0, \qquad u\mid_{E}  = 0, \qquad  \partial_t u\mid_{E} = \varphi 1_F.
\end{equation*} 
We deduce from~\eqref{jerison-lebeau} and~\eqref{log-mal} 
\begin{multline}
\| \phi \|_{L^2(M)}\leq C e^{D\Lambda} \| \phi\|_{L^2(U_j)}\leq C' e^{D\Lambda}\sup_{U_j} | \phi(x)| \\
\leq  C' e^{D\Lambda} \sup_{\Omega} |\nabla_{t,x} u| \leq C'' e^{D\Lambda} \bigl( \sup_E |\nabla_{t,x} u| \bigr)^{\alpha} \bigl( \sup_{\mathcal{K}} |\nabla_{t,x} u|\Bigr)^{1- \alpha}\\
\leq C'' e^{D'\Lambda} \bigl( \sup_F |\phi| \bigr)^{\alpha} \bigl( \| \phi \|_{L^2}\bigr)^{1- \alpha},
\end{multline}
which implies
$$\| \phi \|_{L^2(M)}^\alpha\leq  C'' e^{D'\Lambda} \bigl( \sup_F |\phi| \bigr)^{\alpha} $$
Another use of Sobolev embeddings allows to conclude the proof {{\eqref{BORNERAF}}}.

\section{The double manifold}\label{sec.3}
In this section we give the proof of Theorem~\ref{spectral} for a manifold with boundary $M$ and Dirichlet or Neumann boundary conditions on $\partial M$. The classical idea is to reduce this question to the case of a manifold without boundary by gluying two copies of $M$ along the boundary in such a way the new double manifold $\widetilde{M}$ inherits a Lipschitz metric, which allows to apply the previous results (without boundary) to this double manifold. However, this procedure of gluying has to be done properly, as otherwise the resulting {{glued}} metric might not even be continuous. The main difficulty in our context comes from the fact that the usual method for this {\em doubling} procedure relies on the use of a reflexion principle in geodesic coordinate systems. However, the existence of such coordinate systems requires at least $C^2$ (resp. $C^3$) regularity for the metric (resp. the domain), to be compared with our $W^{1,\infty}$ and $W^{2, \infty}$ assumptions, to get a $C^1$ (hence integrable) geodesic flow. To circumvent this technical difficulty, we shall define a \textit{pseudo-geodesic system} relying on a regularisation of the normal direction to the boundary, which will be $W^{2, \infty}$ and tangent at the boundary to the ``geodesic coordinate system" (which actually does not exist at this low regularity level).

Let $\widetilde{M} = \overline{M} \times \{-1,1\}/\partial M$  the double space made of two copies of $\overline{M}$ where we identified the points on the boundary, $(x,-1)$ and $(x,1)$, $x\in \partial M$. 
\begin{thm}[The double manifold]\label{double}
Let $g$ be given. There exists a $W^{2, \infty}$ structure on the double manifold $\widetilde{M}$, a metric $\widetilde{g}$  of class $W^{1, \infty}$ on $\widetilde{M}$, and  a density $\widetilde{\kappa}$ of class $W^{1, \infty}$ on $\widetilde{M}$ such that the following holds.
\begin{itemize}
\item The maps
$$ i^\pm :  x\in M  \rightarrow (x, \pm 1) \in \widetilde{M} = M \times \{\pm 1\} / \partial M$$ are isometric embeddings.
\item The density induced on each copy of $M$ is the density $\kappa$,
$$\widetilde{\kappa} \mid_{M\times\{\pm1 \}} = \kappa.
$$
\item For any eigenfunction $e$ with eigenvalue $\lambda^2$ of the Laplace operator $-\Delta = - \frac 1 {\kappa} \text{div } g^{-1} \kappa \nabla$ with Dirichlet or Neumann boundary conditions, there exists an eigenfunction $\widetilde{e}$ with the same eigenvalue $\lambda$ of the Laplace operator $-\Delta = - \frac 1 {\widetilde{\kappa} }\text{div } \widetilde{g}^{-1} \widetilde{\kappa} \nabla$ on $\widetilde{M}$ such that 
\begin{equation}\label{ext}
\widetilde{e} \mid_{M \times \{1\}} = e, \quad \widetilde{e} \mid_{M \times \{-1\}} = \begin{cases}  -e \quad &(\text{Dirichlet boundary conditions}),\\
 e \qquad &(\text{Neumann boundary conditions}). \end{cases}
\end{equation}
\end{itemize}
\end{thm}
\begin{cor}
Estimate~\eqref{BORNERAF} for manifolds without boundaries implies~\eqref{BORNE} for Dirichlet or Neumann boundary conditions, and in the case of Dirichlet boundary conditions, we could even add any constant to the spectral projector and replace $\phi$ by 
$$ \Psi = u_0 + \sum_{\lambda_k \leq \Lambda} u_k e_k (x).$$
\end{cor}
\begin{rem}\label{rema.eigen}
Since the vector spaces generated respectively by the  Dirichlet or Neumann eigenfunctions are dense in $L^2(M)$, the vector space generated by their extensions as defined in~\eqref{ext} is dense in $L^2( \widetilde{M})$. We deduce that there exists a Hilbert basis of $L^2(\widetilde{M})$ made of eigenfunctions of $\widetilde{\Delta}$ on $\widetilde{M}$ which are the extensions of the Dirichlet and the Neuman eigenfunctions  of $\Delta$ on $M$.
\end{rem}
To prove Theorem~\ref{double}, we are going to endow $\widetilde{M}$ with a $W^{2, \infty}$ manifold structure and a Lipschitz metric $\widetilde{g}$ which coincides with the original metric $g$ on each copy of $M$. For this  we just need to work near the boundary $\partial M$ (as away from $\partial M$, $\widetilde{M}$ coincides with one of the copies $M\times \{\pm 1\}$). 
Consider a point $x_0 \in \partial M$. There exists a covering $\partial M \subset \cup_{j=1}^N U_j $,  (here $\partial M$ is seen as a subset of $\overline{M}$) where $U_j$ are open sets of $\overline{M}$ and $W^{2, \infty}$ diffeomorphisms 
$$ \psi_j: V_j \rightarrow \mathbb{R}^d= \mathbb{R}_y \times \mathbb{R}^{d-1}_x, \qquad \overline{U_j } \subset V_j \subset \overline{M}$$ such that $\psi_j (V_j) = B(0, 1)_{y,x}\cap \{ y \geq 0\}$ and $\psi_j (U_j) \subset [0, \epsilon] \times B(0,\delta)_x$ for some $\delta, \epsilon > 0$ small enough.  Here $W^{2, \infty}$ regularity means $W^{2, \infty}\cap C^1$ regularity of all the change of charts 
$$ \psi_k \circ \psi_j ^{-1}:  \psi_j (U_k \cap U_j) \rightarrow \mathbb{R}^d.$$

Let $a = a(y,x)$ be the metric in this coordinate system, which is hence $W^{1, \infty}$ and defined for 
$$\|x\| \leq \delta ', \quad  y \in [0, \epsilon'], \quad \delta' < \delta <1, \quad \epsilon' < \epsilon.$$
For any $x \in \{ y = 0\}$, consider the vector defined by 
\begin{equation}\label{normale}
 n(x) = \bigl(\lambda(x) \bigr)^{-1/2}  a^{-1} (0,x) \begin{pmatrix} 1 \\0 \end{pmatrix}, \qquad \textrm{for} \quad \lambda (x) = (1,0) \cdot \,  a^{-1} (0,x) \begin{pmatrix}  1 \\0 \end{pmatrix}. \qquad 
 \end{equation} One can check that
 \begin{equation}
x \mapsto   n(x) \in W^{1, \infty}(B(0,\delta'))
\label{eq:regularite de la normale}
 \end{equation}
is the inward normal to the boundary for the metric $a$ at the point $(0,x) \in \psi_j(\partial M \cap V_j)$. Indeed, 
\begin{equation*}
\ ^tn(x) a(0,x) n(x) = \bigl(\lambda(x)\bigr)^{-1} (1,0) a^{-1}(0,x) a(0,x) a^{-1}(0,x) \begin{pmatrix} 1 \\0 \end{pmatrix}=1,
\end{equation*} which makes $n(x)$ unitary and if $X \in \R^{d-1}_x$, then
\begin{equation*}
\ ^t n(x) a(0,x)\begin{pmatrix} 0 \\X \end{pmatrix}= \bigl(\lambda(x) \bigr)^{-1/2}(-1,0)  a^{-1}(0,x) a(0,x)\begin{pmatrix} 0 \\X \end{pmatrix}=0,
\end{equation*} which proves that $n(x)$ is orthogonal to the vectors tangent to the boundary. Finally, since its first component is positive, $n(x)$ points inward. 

We now study the regularity of the quasi-geodesic coordinates. Let $\chi \in C^\infty_0 ( B(0, \delta'))$ be equal to $1$ in $B(0,\delta)$, and 
$$m(s,z) = e^{-\langle sD_z\rangle +1}(\chi n) (0,z), \qquad \textrm{with } \langle z\rangle = (1+ |z|^2)^{1/2}.$$
\begin{lem}\label{bornelinf}
For any $q\in \R$, $j=1, \dots, d-1$, the operators 
$$\langle sD_z\rangle ^{q}e^{-\langle sD_z\rangle +1} \text{ and }sD_{z_j}\langle sD_z\rangle ^{q}e^{-\langle sD_z\rangle +1} $$ are uniformly bounded on $L^\infty (\mathbb{R}^{d-1})$, with respect to the parameter $s\in \mathbb{R}$.

\end{lem}

\begin{proof}
Indeed, these are convolution operators with kernels
$$K_{1,q}(z) = \frac 1 {s^{d-1}} \mathcal{F} \Bigl(\langle \xi \rangle ^{q}e^{-\langle \xi \rangle+1}\Bigr) ( \frac z s),  \qquad K_{2,q}=  \frac 1 {s^{d-1}} \mathcal{F} \Bigl( i\xi_j \langle \xi \rangle ^{q}e^{-\langle \xi \rangle+1}\Bigr) ( \frac z s),$$
where $\mathcal{F}$ stands for the usual Fourier transform on $\R^{d-1}$. Since the functions
\begin{equation}
\xi \mapsto \langle \xi \rangle ^{q}e^{-\langle \xi \rangle}, \qquad \xi \mapsto \xi_j \langle \xi \rangle ^{q}e^{-\langle \xi \rangle}
\end{equation}
are in the Schwartz class, we deduce that  the kernels $K_{1,q}$ and $K_{2,q}$ are uniformly bounded (with respect to $s\geq 0$) in $L^1 (\mathbb{R}^{d-1}_x)$, which implies that the corresponding operators are bounded on $L^\infty(\mathbb{R}^{d-1}_x)$ (uniformly with respect to $s\geq 0$).
\end{proof}

According to (\ref{eq:regularite de la normale}), the map $(s,z) \mapsto m(s,z)$ is Lipschitz and therefore the map 
$z\mapsto m(0,z)$ is also Lipschitz. 
We deduce using Lemma~\ref{bornelinf} and the basic relations 
$$ \frac{d}{ds} \langle sD_z\rangle =  s D_z^2\langle sD_z\rangle^{-1}, \qquad \frac{d}{ds} e^{-\langle sD_z\rangle +1} =  -s D_z^2\langle sD_z\rangle^{-1} e^{-\langle sD_z\rangle +1},$$
that the map
$$ (s,z) \in [- \epsilon' , \epsilon ']\times B(0, \delta') \rightarrow  \phi_j (s,z) = z+s m(s,z)= z + s e^{-\langle sD_z\rangle +1} (\chi n\mid_{s=0}) $$
is  $W^{2, \infty}$. Indeed, since by assumption $(\chi n)\in L^\infty$ and $\nabla_z ( \chi n)\in L^\infty$, a direct calculation gives using Lemma~\ref{bornelinf},
\begin{equation}
\nabla_z \phi_j (s, z) = {1}+ s e^{-\langle sD_z\rangle +1} (\nabla_z (\chi n\mid_{s=0})) 
  \in L^\infty( (- \epsilon,\epsilon) \times B(0,\delta)_z).
\end{equation}
\begin{multline}
\partial_s \phi_j (s,z) =   e^{-\langle sD_z\rangle +1} (\chi n\mid_{s=0}) - s^2 D_z^2 \langle sD_z\rangle^{-1} e^{-\langle sD_z\rangle +1} (\chi n\mid_{s=0})\\
=(1-\langle sD_z\rangle+\langle sD_z\rangle^{-1})e^{-\langle sD_z\rangle +1} (\chi n\mid_{s=0}) \in L^\infty( (- \epsilon,\epsilon) \times B(0,\delta)_z).
\end{multline}
and 
\begin{equation}
 \nabla_z^2 \phi_j (s,z)= s \nabla_z e^{-\langle sD_z\rangle +1} (\nabla_z (\chi n\mid_{s=0})) \in L^\infty( (- \epsilon,\epsilon) \times B(0,\delta)_z).
 \end{equation}
 \begin{multline}
 \partial_s \nabla_z \phi_j (s,z) = (1-\langle sD_z\rangle+\langle sD_z\rangle^{-1})e^{-\langle sD_z\rangle +1} (\nabla_z (\chi n\mid_{s=0}))  \\ 
 \in L^\infty( (- \epsilon ,\epsilon) \times B(0,\delta)_z).
 \end{multline}
 \begin{equation}
 \begin{aligned} \partial_s^2  \phi_j (s,z) 
 &= -  \bigl(1-\langle sD_z\rangle+\langle sD_z\rangle^{-1}\bigr) sD_z^2 \langle sD_z\rangle ^{-1}e^{-\langle sD_z\rangle +1} (\chi n\mid_{s=0})\\ 
& \qquad  - sD_z^2 \bigl( \langle sD_z\rangle ^{-1} + \langle sD_z\rangle ^{-3}\bigr)e^{-\langle sD_z\rangle +1}  (\chi n\mid_{s=0}) \\
 &= sD_z^2 \bigl( 1-  2\langle sD_z\rangle ^{-1}-  \langle sD_z\rangle ^{-2}-\langle sD_z\rangle ^{-3}\bigr)e^{-\langle sD_z\rangle +1}  (\chi n\mid_{s=0}) \\
 &= \sum_{p=1}^{d-1} sD_{z_p} \bigl( 1-  2\langle sD_z\rangle ^{-1}-  \langle sD_z\rangle ^{-2}-\langle sD_z\rangle ^{-3}\bigr)e^{-\langle sD_z\rangle +1}  (D_{z_p}(\chi n\mid_{s=0})) \\
 & \qquad \hbox to 6cm{}   \in L^\infty( (- \epsilon ,\epsilon) \times B(0,\delta)_z).
 \end{aligned}
 \end{equation}
 
  The differential of $\phi_j$ at $s=0$ is
  $$ d_{s,z} \phi_j \mid_{s=0} = \begin{pmatrix} \chi n(0,z)_y & 0 \\ \chi n(0,z)_z &\text{Id} \end{pmatrix},$$ 
  which, according to~\eqref{normale} and the fact that $a$ is definite positive, is invertible for  $z \in B(0, \delta)$. Hence, we deduce that $\phi_j$ is a $W^{2, \infty}$ diffeomorphism from a neighborhood of  $\{0\}_s \times B(0, \delta)_z$ to a neighborhood of 
  $\{0\}_s \times B(0, \delta)_z$.  Notice also that since $\phi_j$ sends the half plane $\{ s>0\}$ to itself, its inverse also sends the half plane $\{ s>0\}$ to itself.
  As a consequence, shrinking $U_j$ into a possibly smaller $U'_j$ we get a covering 
  $$\partial M \subset \cup_{j=1}^N U'_j$$ and $W^{2, \infty}$ diffeomorphisms $\psi'_j = \phi_j^{-1} \circ \psi_j$
  such that after this change of variable, the metric $b(s,z)$ is given for $s\geq 0$ by 
  $$b(s,z) =  ^t \psi_j' a(y,z) \psi_j' .$$
  In particular, for $s=0^+$  we get 
   \begin{equation}\label{metrique-bord}
   b(0^+,z) = \begin{pmatrix}  n(z)_y &  n(z)_z \\ 0 &\text{Id} \end{pmatrix} a(0,z) \begin{pmatrix}  n(z)_y & 0 \\  n(z)_z &\text{Id} \end{pmatrix}.  
   \end{equation}
   Since $n(z) $ is the normal to the boundary we have 
   $$ ^t n(z) a(0,z) n(z) =1, \qquad \begin{pmatrix} 0, Z\end{pmatrix} a(0,z) n(z) =0, \quad  \forall Z \in \mathbb{R}^{d-1}.   $$
  We deduce
   \begin{equation}\label{normal}
     b(0^+,z)   = \begin{pmatrix}  1& 0 \\  0&b'(z) \end{pmatrix},
   \end{equation}
   with $b'(z)$ positive definite. 
   We just proved,
   \begin{prop}\label{geodesic}
Assume that $M$ is a  $W^{2, \infty}$  manifold of dimension $d$ with boundary, endowed with a Lipschitz (definite positive) metric $g$ and a Lipschitz (positive) density $\kappa$. Let
\begin{equation}\label{delta'}
 \Delta = \frac 1 {\kappa} \text{div} g^{-1}(x) \kappa \nabla_x= \text{div}_\kappa \nabla_g. 
 \end{equation} 
 Then near any point of the boundary $X_0 \in \partial M$ there exists  a $W^{2, \infty}$ coordinate system such that in this coordinate system 
 \begin{equation}\label{3.13}
 \begin{gathered}
 X_0 = (0,0) \in \mathbb{R}_y\times \mathbb{R}_z^{d-1}, \quad  \Omega = (0, + \infty) \times \mathbb{R}^{d-1}, \quad  \partial \Omega = \{0\} \times \mathbb{R}^{d-1}\\
 \Delta = \frac 1 { \kappa (y,z) }  \ ^t \nabla_{y,z}  \kappa (y,z) b(y,z) \nabla_{y,z}, \qquad  b \mid_{\partial \Omega} = \begin{pmatrix}  1& 0 \\  0&b'(z) \end{pmatrix}.
 \end{gathered}
 \end{equation}
 \end{prop}
   \begin{rem} In a geodesic coordinate system, we would have a diagonal form for the metric as in~\eqref{normal} in a neighborhood of the boundary. Proposition~\ref{geodesic} corresponds to the fact that our coordinate system is at the boundary ``tangent to a geodesic coordinate system". 
   \end{rem}
   Summarizing, we have defined a covering of $\partial M \subset \cup_{j=1}^N U'_j$, and  $W^{2, \infty}$ diffeomorphisms  
   $$ \psi'_j: U'_j \rightarrow \mathbb{R}^d= \mathbb{R}_s \times \mathbb{R}^{d-1}_x, $$ such that $\psi'_j (V_j)  \subset B(0, 1)_{s,x}\cap \{ s \geq 0\}$, and after the change of variables $\psi'_j$, the metric takes the form~\eqref{normal} on the boundary $\{s=0\}$.
   
   We can now perform the gluying by defining a covering of $\partial M$ (now seen as a subset of $ \widetilde{M}$),
   $$ \partial M \subset \cup_{j=1}^N  U'_j \times \{-1, 1\} = \cup_{j=1}^N \widetilde{U}_j$$
   where we identify the points in $U'_j \cap \partial M  \times \{-1, 1\}$, and define the map 
   $$ \Psi_j :  z \in U'_j \times \{ \epsilon \} \mapsto  \begin{cases} \psi'_j (x)  & \text{ if } \epsilon =1, \\ S \circ \psi'_j (x) & \text{ if } \epsilon =-1, \end{cases}
   $$ where 
  $$ S(s,z) = (-s,z).$$
  To conclude the proof of the first part of Theorem~\ref{double}, it remains to check that 
  \begin{itemize}
  \item The image of the metric induced on $\widetilde{M}$ by the metrics on the two copies of $\overline{M}$ is well defined and Lipschitz,
  \item The change of charts 
  $$\Psi_k \circ \Psi_j^{-1} : \Psi_j (\widetilde{U_k} \cap \widetilde{U_j}) \rightarrow \mathbb{R}^d$$
  are $W^{2, \infty}$.
  \item The density $\widetilde{\kappa} $ obtained by gluying the two copies of $\kappa$ on each copy of $M$ is $W^{1, \infty}$.
  \end{itemize}
  The first result follows from~\eqref{normal} because on  $\Psi_j (U'_j \times\{1\})$ the metric is given  by 
  $$b(s,x)1_{s\geq0},$$
  while on $\Psi_j (U'_j \times\{-1\})$, it is given by 
 \begin{equation}
 \label{metric}
   S' b\circ S(s,x) S' 1_{s\leq 0} = \begin{pmatrix}  -1& 0 \\  0&\text{Id}   \end{pmatrix}\begin{pmatrix}  1&  \ ^t r(s,x) \\  r(s,x)&b' (-s, x)  \end{pmatrix}\begin{pmatrix}  -1&0 \\   0&\text{Id}  \end{pmatrix}1_{s\leq 0},
 \end{equation}
 where from~\eqref{3.13}, $  r(0,x) =0$.
 
 As a consequence, the two metrics coincide on $\{ s=0\}$ and they define a Lipschitz metric on $(- \epsilon ' , \epsilon) \times B(0, \delta)$.
 To check the $W^{2, \infty}$ smoothness of the change of charts, we write
 $$ \Psi_k \circ \Psi_j^{-1} = \begin{cases} \psi'_k \circ( \psi'_j)^{-1} &\text{ on }  \{ s\geq 0\}, \\ 
 S \circ  \psi'_k \circ( \psi'_j)^{-1} \circ S &\text{ on }  \{ s\leq 0\}.
 \end{cases}
 $$
  Taking derivatives we get 
  \begin{equation}
  \label{derivative}
   d_{s,z} \Psi_k \circ \Psi_j^{-1} = \begin{cases} d_{s,z} (\phi'_k \circ( \phi'_j)^{-1}) &\text{ on }  \{ s> 0\}, \\ 
\begin{pmatrix}  -1& 0 \\  0&\text{Id}   \end{pmatrix} d_{s,z}    (\phi'_k \circ( \phi'_j)^{-1} )\begin{pmatrix}  -1& 0 \\  0&\text{Id}   \end{pmatrix} &\text{ on }  \{ s< 0\}. 
 \end{cases}
\end{equation}
  We now remark that by construction the differential $d_{y,x} \phi' \mid_{\partial M}$ sends the normal to the boundary to the normal to the boundary $\begin{pmatrix} 1\\0\end{pmatrix} $ and sends all  vectors tangent to the boundary  to tangent vectors $\begin{pmatrix} 0\\ Z' \end{pmatrix}$. As a consequence 
  \begin{equation}\label{derivativebis}d_{s,z} (\phi'_k \circ( \phi'_j)^{-1}) \mid_{\{ s=0\} }= \begin{pmatrix} 1 & 0 \\ 0 & \text{q(z)} \end{pmatrix}. 
  \end{equation}
  We deduce from~\eqref{derivativebis} that the two limits of the differentials $s\rightarrow 0^+$ and $s\rightarrow 0^-$ coincide 
   \begin{equation} \label{cont-deriv}
    d_{s,z} \Psi_k \circ \Psi_j^{-1}\mid_{s=0^+}  = d_{s,z} \Psi_k \circ \Psi_j^{-1}\mid_{s=0^-},
\end{equation}
and consequently, the differential is $C^0$. 
Let us now study the $L^\infty$ boundedness of derivatives of order $2$. The case of space derivatives  $d^2_{z,z}$ or $d^2_{s,z}$ is easy because we just have to take an additional tangential derivative $d_z$ in~\eqref{cont-deriv}. 
Such derivatives are tangent to the boundary $\{ s=0\}$ giving 
 $$ d^2_{z,z} \Psi_k \circ \Psi_j^{-1}\mid_{s=0^+}  = d^2_{z,z} \Psi_k \circ \Psi_j^{-1}\mid_{s=0^-} = d^2_{z,z} \phi'_k \mid_{\partial M} \circ (\phi'_j)^{-1} \mid_{\{ s=0\}}.
$$
  Finally, the case of $d^2_{s,s}$ derivative follows from the jump formula and the use of~\eqref{cont-deriv} which  shows that the first order derivatives have no jump, as
  \begin{multline}
  \frac{\partial^2} {\partial s^2}  \Bigl(\psi'_k \circ( \psi'_j)^{-1} 1_{  s> 0} + S_k \circ  \psi'_k \circ( \psi'_j)^{-1} \circ S_j 1_{s<0}\Bigr) \\
  = \frac{\partial^2} {\partial s^2}  \Bigl(\psi'_k \circ( \psi'_j)^{-1} \Bigr) 1_{  s> 0} + \frac{\partial^2} {\partial s^2}\Bigl( S_k \circ  \psi'_k \circ( \psi'_j)^{-1} \circ S_j 1_{s<0}\Bigr)\\+ \Bigl(\frac{\partial} {\partial s}  \bigl(\psi'_k \circ( \psi'_j)^{-1}\mid_{ s=0^+}\bigr)  - \frac{\partial} {\partial s} \bigl(S_k \circ  \psi'_k \circ( \psi'_j)^{-1} \circ S_j \mid_{ s=0^-}\bigr) \Bigr) \otimes \delta_{s=0} \\
  +  \Bigl( \psi'_k \circ( \psi'_j)^{-1}\mid_{ s=0^+}  - S_k \circ  \psi'_k \circ( \psi'_j)^{-1} \circ S_j \mid_{ s=0^-} \Bigr) \otimes \delta'_{s=0}\\
= \frac{\partial^2} {\partial s^2}  \Bigl(\psi'_k \circ( \psi'_j)^{-1} \Bigr) 1_{  s> 0} + \frac{\partial^2} {\partial s^2}\Bigl( S_k \circ  \psi'_k \circ( \psi'_j)^{-1} \circ S_j \Bigr)1_{s<0}.
\end{multline}
The last result for the density $\widetilde{\kappa}$ follows from this $W^{2, \infty}$ regularity of the change of charts.

It remains to prove the second part in Theorem~\ref{double} (about the eigenfunctions). Let $e$ be an eigenfunction of $- \Delta$ on $M$ with Dirichlet condition, associated to the eigenvalue $\lambda^2$. We define
$$ \widetilde{e} (x, \pm 1) = \pm e(x). $$ This definition makes sense because on the boundary $e(x) =0 = - e(x)$. Now we check that $\widetilde{e}$ is an eigenfunction of $ \widetilde{\Delta}$ on $\widetilde{M}$. Away from the boundary $\partial M$ this is clear while near a point $x \in \partial M \subset \widetilde{M}$ we can work in a coordinate chart $(\Psi_j, \widetilde{U}_j)$. In this coordinate chart, the function $\widetilde{e}$ is defined  by 
$$ \widetilde{e} (s,z) = \begin{cases} e(s,z) &\text{ if } s>0,\\
- e(-s,z) &\text{ if } s<0.
\end{cases}
$$
In $\{ \pm s >0\} $,  $\widetilde{e} $ satisfies $- \widetilde{\Delta} \widetilde{e} = \lambda^2 e$, and near $\partial M$ in our coordinate systems, we have  
\begin{equation}
\widetilde{e} (s,z) = e(s, z) 1_{s>0} - e(-s,z)1_{s<0}= e(s,z) 1_{s>0} + f(s,z) 1_{s<0},\qquad f(s,z) = - e(-s,z),
\end{equation}
and we have 
\begin{equation}\label{prem}
\begin{aligned}
 \nabla_x (\widetilde{e} )(s,x)&= (\nabla_x e) 1_{s>0} + (\nabla_x f) 1_{s<0},\\
  \partial_s (\widetilde{e} )(s,x)&= (\partial_s e) 1_{s>0} + (\partial_s f) 1_{s<0} + ( e(0^+,z)- f(0^-, z) )\otimes \delta_{s=0}\\
  &= (\partial_s e) 1_{s>0} + (\partial_s f) 1_{s<0} ,
  \end{aligned}
  \end{equation}
  where we used that according to Dirichlet boundary condition $e(0^+,z)= f(0^-, z)=0$. Now, according to~\eqref{normal}, we get 
  $$ \widetilde{b}(s,z)= \begin{pmatrix} b_{1,1}(s,z) & r(s,z) \\ ^t r(s,z) & b'(s,z) \end{pmatrix}
  $$
  with $b_{1,1}(0,z) =1$, $r(0,z) =0$, and we deduce from~\eqref{prem} and the jump formula
  \begin{multline}
  -\widetilde{\Delta} (\widetilde{e} )(s,x))= \lambda^2 \widetilde{e} (s,x) \\
  + \frac 1 {\kappa(0,z)}b_{1,1}(0,z) \Bigl( (\partial_s e) (0^+,z)- (\partial_s f) ( 0^-,z) + r(0,z) \bigl((\nabla_x e) (0^+,z) - (\nabla_x f) (0^-,z)\bigr) \Bigr) \otimes \delta_{s=0}
 \\ = \lambda^2 \widetilde{e} (s,x) 
  \end{multline}
  where here we used $r(0,z)=0$ and that since $f(s,z) = -e(-s,z)$ we have 
  $$ \partial_s f(0^-,z)= \partial_s e(0^+,z).$$
  This ends the proof of Theorem~\ref{double} for Dirichlet boundary conditions. The proof in the case of Neumann boundary conditions is similar by defining 
  $$ \widetilde{e} (s,z) = \begin{cases} e(s, z) &\text{ if } s>0\\
e(-s,z) &\text{ if } s<0.
\end{cases}
$$

  \section{Propagation of smallness for the heat equation} \label{sec.4}
  In this section we show how the first parts in Theorems~\ref{control} and~\ref{control-haus} (i.e. estimates~\eqref{obs} and~\eqref{obsbis}) follow from Theorem~\ref{spectral}. Here we follow closely~\cite[Section 2]{AEWZ14}, which in turn relied on a mixing of ideas from~\cite{mi10}, interpolation inequalities and the telescopic series method from~\cite{PhWa13}. Indeed Theorem~\ref{interpo1} is actually slightly more general than~\cite[Theorem 5]{AEWZ14}, as the constants do not depend on the distance to the boundary but only on the Lebesgue measure of $E$, and the interpolation exponent ($1-\epsilon $ below) can be taken arbitrarily close to $1$. 
The first step is to  deduce  interpolation inequalities from Theorem~\ref{spectral}.

\begin{thm}\label{interpo1}[\protect{compare with \cite[Theorem 6]{AEWZ14}}] Let $\epsilon \in (0,1)$ and $m >0$. Assume that $|E_1|\geq m$, or $\mathcal{C}_{\mathcal{H}}^{d- \delta} (E_2)>m$. Then there exist $N>0, C>0$ such that for all  $0 \leq s <t$, 
\begin{equation}\label{interpopo1}
\| e^{t\Delta } f \| _{L^2 (M)} \leq N e^{\frac {N}{t-s}}  \| e^{t \Delta } f \| _{L^1(E_1)}  ^ {1- \epsilon} \| e^{s \Delta } f\|_{L^2(M)} ^{\epsilon},  
\end{equation} and
\begin{equation}\label{interpopo2}
\| e^{t\Delta } f \| _{L^2 (M)} \leq N e^{\frac {N}{t-s}}  \| e^{t \Delta } f \| _{L^\infty(E_2)} ^ {1- \epsilon}  \| e^{s \Delta } f\|_{L^2(M)} ^{\epsilon}.  
\end{equation}
\end{thm}
\begin{cor}\label{cor44bis}
 Let $m >0$. Assume that $|E_1|\geq m$, or $\mathcal{C}_{\mathcal{H}}^{d- \delta} (E_2)>m$. Then for any $D, B\geq 1$ there exists $A>0, C>0$ such that that for all $0< t_1 <t_2\leq T$, 
 \begin{equation}\label{eq46bis}
 e^{- \frac{A}{(t_2 -t_1)}} \| e^{t_2 \Delta} f \|_{L^2(M)} - e^{- \frac{DA}{(t_2- t_1)}} \| e^{t_1 \Delta} f \|_{L^2(M)} \leq C e^{- \frac{B}{(t_2 -t_1)}})\| e^{t_2\Delta} f \|_{L^1(E_1)}, 
 \end{equation}
 \begin{equation}\label{eq46ter}
 e^{- \frac{A}{(t_2 -t_1)}} \| e^{t_2 \Delta} f \|_{L^2(M)} - e^{- \frac{DA}{(t_2- t_1)}} \| e^{t_1 \Delta} f \|_{L^2(M)} \leq Ce^{- \frac{B}{(t_2 -t_1)}} \| e^{t_2\Delta} f \|_{L^\infty(E_2)}, 
 \end{equation}
or for any  $J$, $|J \cap (t_1, t_2)| \geq \frac{ (t_2 -t_1)} 3$, 
    \begin{equation}\label{eq46inf}
 e^{- \frac{A}{(t_2 -t_1)}} \| e^{t_2 \Delta} f \|_{L^2(M)} - e^{- \frac{DA}{(t_2- t_1)}} \| e^{t_1 \Delta} f \|_{L^2(M)} \leq C \int_{t_1} ^{t_2} 1_{J}(s) \| e^{s\Delta} f \|_{L^1(E_1)} ds,
 \end{equation}

 \begin{equation}\label{eq46}
 e^{- \frac{A}{(t_2 -t_1)}} \| e^{t_2 \Delta} f \|_{L^2(M)} - e^{- \frac{DA}{(t_2- t_1)}} \| e^{t_1 \Delta} f \|_{L^2(M)} \leq C \int_{t_1} ^{t_2} 1_{J}(s) \| e^{s\Delta} f \|_{L^\infty(E_2)} ds.
 \end{equation}
  \end{cor}

\begin{proof} Let us first prove Corollary~\ref{cor44bis} from Theorem~\ref{interpo1} adapting ~\cite{AEWZ14}. 
Let $A>0$. From Theorem~\eqref{interpo1} we get, using Young inequality $ab \leq (1- \epsilon)a^{1/(1- \epsilon)} + \epsilon b^{\epsilon ^{-1}}$, 
 \begin{multline}
 e^{-\frac{A}{(t_2- t_2)} }  \| e^{t_2 \Delta} u\|_{L^2(M)} \leq N e^{\frac{N-A}{(t_2-t_1)} }\| e^{t_2\Delta} f \|_{L^\infty(E_1)}^{(1- \epsilon)} \| e^{t_1 \Delta} f \|_{L^2(M)}^{\epsilon}\\
 \leq Ne^{-\frac{A}{2(t_2- t_2)} }\| e^{t_2\Delta} f \|_{L^\infty(E_1)}^{(1- \epsilon)}    e^{\frac{N-A/2}{(t_2-t_1)} } \| e^{t_1 \Delta} f \|_{L^2(M)}^{\epsilon}\\
 \leq N^ {(1- \epsilon)^{-1}}(1- \epsilon) e^{-\frac{A}{2(1- \epsilon)(t_2- t_2)} } \| e^{t_2\Delta} f \|_{L^\infty(E_1)} +  \epsilon e^{ \frac{N-A/2}{\epsilon (t_2-t_1)}} \| e^{t_1 \Delta} f \|_{L^2(M)}
 \end{multline}
and~\eqref{eq46bis} follows from choosing $2\epsilon < D^{-1} $ in Theorem~\ref{interpo1} and then 
 $$ A\geq 2 B, \qquad \frac{A/2- N} \epsilon >DA .$$
 The proof of ~\eqref{eq46ter} is similar. Let us now turn to the proof of~\eqref{eq46}. 
  From the assumption $|J \cap (t_1, t_2)| \geq \frac{ (t_2 -t_1)} 3$, we deduce
 \begin{equation}\label{minor}
 |J \cap (t_1+ \frac{ t_2 - t_1} 6, t_2)| \geq \frac{ (t_2 -t_1)} 6.
 \end{equation}
 Now, we have from~\eqref{interpopo2}, for $ t\in (t_1+ \frac{ t_2 - t_1} 6, t_2)$, 
 \begin{multline}
  \| e^{t_2\Delta } f \| _{L^2 (M)} \leq \| e^{t\Delta } f \| _{L^2 (M)} \leq  N e^{\frac {N}{t-t_1}} \| e^{t \Delta } f \| _{L^1(E_2)} ^ {1- \epsilon} \| e^{t_1 \Delta } f\|_{L^2(M)} ^{\epsilon}  \\
  \leq  N e^{\frac {6N}{t_2-t_1}} \| e^{t \Delta } f \| _{L^1(E_2)}  ^ {1-\epsilon} \| e^{t_1 \Delta } f\|_{L^2(M)} ^{\epsilon}. 
  \end{multline}
 Integrating this inequality on $J \cap(t_1+ \frac{ t_2 - t_1} 6, t_2)$ and using H\"older inequality gives  
 \begin{multline}
  |J \cap (t_1+ \frac{ t_2 - t_1} 6, t_2)|  \| e^{t_2\Delta } f \| _{L^2 (M)} \\
  \leq N e^{\frac {6N}{t_2-t_1}} \Bigl(\int_{t_1+ \frac{ t_2 - t_1} 6} ^{t_2} 1_{J}(t)\| e^{t \Delta } f \| _{L^1(E)} dt\Bigr)^{1- \epsilon}  \| e^{t_1 \Delta } f\|_{L^2(M)} ^{\epsilon},
  \end{multline}
which using~\eqref{minor} (and replacing $6N$ by $6N+1$) gives
 \begin{equation}
  \| e^{t_2\Delta } f \| _{L^2 (M)} \\
  \leq N e^{\frac {6N+1}{t_2-t_1}} \Bigl(\int_{t_1+ \frac{ t_2 - t_1} 6} ^{t_2} 1_{J}(t)\| e^{t \Delta } f \| _{L^1(E)} dt\Bigr)^{1- \epsilon} \| e^{t_1 \Delta } f\|_{L^2(M)} ^{\epsilon}.
  \end{equation}
The rest of the proof of~\eqref{eq46} follows now the same lines as the proof of~\eqref{eq46ter}. Finally the proof of~\eqref{eq46inf} is similar.
 \end{proof}
\begin{rem} \label{rem4.2}The proof above shows that in~\eqref{eq46inf} and~\eqref{eq46}, we can replace the sets $E_1$, $E_2$ by sets $E_1(t)$, $E_2(t)$ if we assume that $|E_1(t)|\geq m $ or $\mathcal{C}^{d- \delta}_{\mathcal{H}}(E_2(t)) \geq m$ {\em uniformly} with respect to $s\in I$, so that we can apply~Theorem~\ref {interpo1} with sets $E_1(t)$ and $E_2(t)$.
\end{rem}
\begin{proof}[Proof of Theorem~\ref{interpo1}]
Let $0 \leq s < t$ and for $f \in L^2(M)$ let
$$ f = \Pi_\Lambda f + \Pi^\Lambda f,$$ where $\Pi_\Lambda$ is the orthogonal projector on the vector space generated by $\{ e_k; \lambda _k \leq \Lambda\}$.  
We have 
\begin{equation}\label{interpo5}
 \begin{aligned}  \| e^{t \Delta } f \|_{L^2(M)}& \leq \| e^{t \Delta }\Pi_\Lambda f\| _{L^2(M)} + \|  e^{t \Delta } \Pi^\Lambda f\| _{L^2(M)}  \\
 & \leq N e^{N \Lambda } \| e^{t \Delta }  \Pi_\Lambda f\| _{L^1(E_1)} + \|  e^{t \Delta } \Pi^\Lambda f\| _{L^2(M)}\\
 & \leq N e^{N \Lambda } \Bigl(\| e^{t \Delta }  f\|_{L^1(E_1)} +\|e^{t \Delta} \Pi^\Lambda f\|_{L^2(M)} \Bigr) + \|  e^{t \Delta } \Pi^\Lambda f\| _{L^2(M)}  \\
 & \leq (N+1) e^{N \Lambda } \Bigl(\| e^{t \Delta }  f\| _{L^1(E_1)} +e^{- \Lambda^2 (t-s)}\| e^{s \Delta} \Pi^\Lambda f\|_{L^2(M)} \Bigr)\\
 & \leq (N+1) e^{N \Lambda } \Bigl(\| e^{t \Delta }  f\| _{L^1(E_1)} +e^{- \Lambda^2 (t-s)}\| e^{s \Delta}  f\|_{L^2(M)} \Bigr),\\
  \end{aligned}
 \end{equation}

 Since 
 $$ \sup _{\Lambda \geq 0} e^ {N \Lambda - \epsilon{\Lambda^2}  (t-s)} = e^{\frac{N^2} {4\epsilon (t-s)}},$$
 we deduce
 \begin{equation}
 \label{interpo3}
  \| e^{t \Delta } f \|_{L^2(M)} \leq (N+1) e^{\frac{N^2} {4\epsilon (t-s)}} \Bigl( e^{\epsilon {\Lambda^2}  (t-s)} \| e^{t \Delta }  f\| _{L^1(E_1)}  + e^{- (1- \epsilon) {\Lambda^2} (t-s)} \|e^{s \Delta}  f\|_{L^2(M)}\Bigr).\\
 \end{equation}
 Since $\Lambda$ is a free parameter, and $(t-s)>0$, we can minimize the r.h.s. of~\eqref{interpo3} with respect to the parameter $\alpha = e^{-\frac{\Lambda^2} 2 (t-s) } \in (0,1)$, by choosing
 $$ e^{\Lambda^2 (t-s) } = \frac {\|e^{s \Delta}  f\|_{L^2(M)}}{\| e^{t \Delta }  f\| _{L^1(E_1)}},$$
 which gives 
\begin{equation}
  \| e^{t \Delta } f \|_{L^2(M)} \leq 2(N+1) e^{\frac{N^2} {4 \epsilon (t-s)}} \Bigl(  \| e^{t \Delta }  f\| ^2_{L^1(E_1)}\Bigr) ^{ 1- \epsilon}  \Bigl(\|e^{s \Delta}  f\|_{L^2(M)}\Bigr)^{\epsilon},
 \end{equation}
and thus ~\eqref{interpopo1} follows.

To prove \eqref{interpopo2} we have to adapt the method. 
We get using Lemma~\ref{sobolev}
\begin{equation}\label{interpo6}
 \begin{aligned}  \| e^{t \Delta } f \|_{L^2(M)}& \leq \| e^{t \Delta }\Pi_\Lambda f\| _{L^2(M)} + \|  e^{t \Delta } \Pi^\Lambda f\| _{L^2(M)}  \\
 & \leq N e^{N \Lambda } \| e^{t \Delta }  \Pi_\Lambda f\| ^2_{L^\infty(E_2)} + \|  e^{t \Delta } \Pi^\Lambda f\| _{L^2(M)}\\
 & \leq N e^{N \Lambda } \Bigl(\| e^{t \Delta }  f\| _{L^\infty(E_2)} +\|e^{t \Delta} \Pi^\Lambda f\|_{\mathcal{H}^{\sigma}(M)} \Bigr) + \|  e^{t \Delta } \Pi^\Lambda f\| _{L^2(M)}  \\
 & \leq (N+1) e^{N \Lambda } \Bigl(\| e^{t \Delta }  f\| _{L^\infty(E_2)} +\| e^{t \Delta} \Pi^\Lambda f\|_{{\mathcal{H}^{\sigma}}(M)} \Bigr).\\
  \end{aligned}
 \end{equation}
Let us study the quantity
$$\| e^{t \Delta}  \Pi^\Lambda  f\|^2_{\mathcal{H}^{\sigma}(M)} = \sum_{\lambda_k > \Lambda}  \Bigl( e^{-2\lambda_k^2 (t-s)}  \lambda_k^{2\sigma} \Bigr) e^{-2 \lambda_k^2 s} |f_k|^2.
$$ 
Since
 $$ \sup _{\lambda_k \geq \Lambda  } e^ {- 2\epsilon{\lambda_k^2}  (t-s)}  \lambda_k^{2\sigma} \leq \bigl( \Lambda + \frac{\sigma} {  \epsilon (t-s) }\Bigr)^{2\sigma},$$
 we deduce 
\begin{multline}
\| e^{t \Delta}  \Pi^\Lambda  f\|^2_{\mathcal{H}^{\sigma}(M)} \leq \bigl( \Lambda + \frac{\sigma} { 2 \epsilon (t-s) }\Bigr)^{2\sigma} \sum_{\lambda_k > \Lambda}  \Bigl( e^{-2(1- \epsilon)\lambda_k^2 (t-s)}  \Bigr) e^{-2 \lambda_k^2 s} |f_k|^2\\
\leq \bigl( \Lambda + \frac{\sigma} { 2 \epsilon (t-s) }\Bigr)^{2\sigma}  e^{-2(1- \epsilon)\Lambda^2 (t-s)}  \| e^{s \Delta } f\|_{L^2}\\
\leq C_{\epsilon, \sigma} e^\Lambda e^{-2(1- 2\epsilon)\Lambda^2 (t-s)}  \| e^{s \Delta } f\|_{L^2}
\end{multline} 
and coming back to~\eqref{interpo6}, we get
  \begin{equation}
 \label{interpo3bis}
  \| e^{t \Delta } f \|_{L^2(M)} \leq (N_{\epsilon, \sigma}) e^{{(N+1)} \Lambda } \Bigl(\| e^{t \Delta }  f\| _{L^\infty(E)} +  e^{-2(1- 2\epsilon)\Lambda^2 (t-s)}  \| e^{s \Delta } f\|_{L^2} \Bigr).\end{equation}
The rest of the proof of Theorem~\ref{interpo1} follows by the same optimization argument as before.
 \end{proof}
 
  Once  Corollary~\ref{cor44bis} is established, the rest of the proof of~\eqref{obs}, \eqref{obsbis}, \eqref{obster} and \eqref{obsterbis} and follows closely~\cite[Section 2]{AEWZ14}. For completeness we recall the proof. 
Let us start with the simpler~\eqref{obster}. From~\eqref{eq46bis}, with $t_1= s_{n+1}, t_2= s_n$,  and $D= \tau^{-1}$ we have
 \begin{equation}
 e^{- \frac{A}{(s_n -s_{n+1})}} \| e^{s_n \Delta} f \|_{L^2(M)} - e^{- \frac{DA}{(s_n - s_{n+1})}} \| e^{s_{n+1} \Delta} f \|_{L^2(M)} \leq C e^{- \frac{B}{(s_n- s_{n+1})}}\| e^{s_n\Delta} f \|_{L^1(E_1)} 
 \end{equation}
 Since $(s_{n+1}- s_{n+2}) \geq \frac{ s_n- s_{n+1}}{D}$, we deduce
  \begin{multline}\label{telescopic}
 e^{- \frac{A}{(s_n -s_{n+1})}} \| e^{s_n \Delta} f \|_{L^2(M)} - e^{- \frac{A}{(s_{n+1} - s_{n+2})}} \| e^{s_{n+1} \Delta} f \|_{L^2(M)}\\
  \leq C e^{- \frac{B}{(s_n- s_{n+1})}}\| e^{s_n\Delta} f \|_{L^1(E_1)} \leq C' e^{- \frac{B-1}{(s_n-s_{n+1})}}\| (s_n - s_{n+1})e^{s_n\Delta} f \|_{L^1(E_1)}. 
 \end{multline}
 Summing the telescopic series~\eqref{telescopic}, and using that 
 $$e^{- \frac{A}{(s_n -s_{n+1})}} \| e^{s_n \Delta} f \|_{L^2(M)}\leq e^{- \frac{A}{(s_n -s_{n+1})}}\| f\|_{L^2(M)} \rightarrow _{n\rightarrow + \infty} 0,$$
 we get (recall that $s_0=T$),
 \begin{multline}
 e^{- \frac{A}{(T-s_{1})}} \| e^{T \Delta} f \|_{L^2(M)}\leq C \sum_{n=0}^{+\infty} e^{- \frac{B-1}{(s_{n} - s_{n+1}) }}(s_n - s_{n+1})\| e^{s_n\Delta} f \|_{L^1(E_1)} \\
 \leq C \sup_n e^{- \frac{B-1}{(s_{n} - s_{n+1})}}\| e^{s_n\Delta} f \|_{L^1(E_1)},
 \end{multline} which proves~\eqref{obster}. The proof of~\eqref{obsterbis} is the same.
 
  To prove~\eqref{obsbis} we need the following Lemma from~\cite{PhWa13} about the structure of density points of sets of positive measure on $(0,T)$. 

 \begin{lem}[\protect{\cite[Proposition 2.1]{PhWa13}}]Let $J$  be a subset of positive measure in $(0,T)$. Let $l$ be a density point of $J$. Then for any $z>1$ there exists $l_1\in (l,T)$ such that the sequence defined by 
 $$ l_{m+1} -l = z^{-m} (l_1-l)$$
 satisfies 
 $$ | J \cap (l_{m+1}, l_m) | \geq \frac{(l_m - l_{m+1}) }3.
$$
 \end{lem}

Now,  we apply this result with $z=2$ and from~\eqref{eq46} with $D=2$ and $t_1= l_{m+1}, t_2= l_{m}$ we get
 \begin{equation}
 e^{-\frac{M} {l_m - l_{m+1}} } \| e^{l_m \Delta} f\|_{L^2(M)} -  e^{-\frac{2M} {(l_m - l_{m+1})} } \| e^{l_{m+1}\Delta }f\|_{L^2(M)}\leq C \int_{l_{m+1}}^{l_m} 1_{J}(s) \| e^{s\Delta} f \|_{L^\infty(E)} ds 
 \end{equation}

 Noticing that $\frac 2 {(l_m - l_{m+1})}=\frac 1 {( l_{m+1} - l_{m+2})}$,  we get
 \begin{equation}\label{telesc}
 e^{-\frac{M} {l_m - l_{m+1}} } \| e^{l_m \Delta} f\|_{L^2(M)} -  e^{-\frac{M} {l_{m+1} - l_{m+2}} } \| e^{l_{m+1}\Delta }f\|_{L^2(M)}\leq C \int_{l_{m+1}}^{l_m} 1_{J}(s) \| e^{s\Delta} f \|_{L^\infty(E)} ds 
 \end{equation}
 summing the telescopic series~\eqref{telesc}, and using that 
$$ \lim_{m\rightarrow + \infty} e^{-\frac{M} {l_{m+1} - l_{m+1}} } =0, $$ we get
 $$e^{-\frac{M} {l_1 - l_{2}} } \| e^{l_1 \Delta} f\|_{L^2(M)} \leq C \int_{l_1}^{l} 1_{J}(s) \| e^{s\Delta} f \|_{L^\infty(E)} ds,  $$
 which (since $T>l_1$) implies~\eqref{obsbis}.
 
 To prove~\eqref{obs}, we need  an elementary consequence of Fubini's Theorem.

 \begin{lem}
 Let $F\subset M\times (0,T)$ a set of positive Lebesgue measure. Working in coordinates, we can assume that $F\subset B(x_0, r_0)\times (0,T)$. 
 For almost every $t \in (0,T)$ the sets 
 $$E_t = F \cap \{ M \times \{t\}\} \qquad \textrm{and} \qquad J = \{ t \in (0,T); | E_t| \geq\frac{|F|} {2T}\}  $$ 
 are measurable and   
 $$|J|\geq \frac{|F|}{2T |B(x_0, r_0) |}.$$
 \end{lem}

	\begin{proof}
	Indeed, from Fubini, 
 $$ |F| = \int_J |E_t| dt + \int_{(0,T) \setminus J} |E_t| dt \leq |J| |B(x_0, r_0)| + \frac{ |F| }{2}. $$
	\end{proof}

 Now, the proof of~\eqref{obs} follows exactly the same lines as the proof of~\eqref{obsbis} above by noticing that~\eqref{interpopo1} will hold for $E=E_t$ with constants that are uniform with respect to $t \in I$ (because then $|E_t| \geq \frac{|F|}{2T}$), see Remark~\ref{rem4.2}.

\section{Control for heat equations on ``very small sets"} \label{sec.5} 
Here we give the proof of the exact controllability parts in Theorems~\ref{control-haus} and~\ref{control-hausbis} (this part in Theorem~\ref{control} is very classical and we shall leave it to the reader). We start with Theorems~\ref{control-haus}. Since $J \subset (0,T)$ has positive Lebesgue measure, it is also the case for $J \cap (\epsilon, T)$ for some $0<\epsilon< T$, and as a consequence, we we can assume $J \subset (\epsilon, T)$.  By subadditivity of the Hausdorff content 
$$ \mathcal{C}^{r}_{\mathcal{H}} \left(\bigcup_{j=1}^{+\infty} A_j \right) \leq \sum_{j=1}^{+\infty}  \mathcal{C}^{r}_{\mathcal{H}} (A_j),$$ with 
$$ A_j = E \cap \{ x\in M, d(x, \partial M) \geq\frac 1 j\}, \qquad j \in \N,$$
we deduce that there exists $j_0$ such that 
$$ \mathcal{C}^{d- \delta}_{\mathcal{H}} (E \cap A_{j_0}) >0$$
 because otherwise we would have 
 $$\mathcal{C}^{d- \delta}_{\mathcal{H}} (E \cap \cup_{j=1}^{+\infty}A_{j_0})= \mathcal{C}^{d- \delta}_{\mathcal{H}}(E \setminus \partial M) =0 \Rightarrow \mathcal{C}^{d- \delta}_{\mathcal{H}}(E) =0.$$
 As a consequence, replacing $E$ by $E\cap A_{j_0}$, we can assume that 
 \begin{equation}
 \label{eq.disjoint}
  \exists \epsilon >0 \qquad \textrm{such that} \qquad  \forall x \in E, \quad d(x, \partial M) > \epsilon.
  \end{equation}
 For $w_0 \in L^2(M)$ let $w= e^{(T-t) \Delta} w_0$ be the solution to the backward heat equation
\begin{equation}\label{eq-obs2}
\begin{gathered}
(\partial_t + \Delta ) w = 0, \qquad u \mid_{t=T} = w_0,\\
w \mid_{\partial M} =0 \text{ (Dirichlet condition) or } \partial_\nu w \mid_{\partial M} =0 \text{ (Neumann condition)}.
\end{gathered}
\end{equation}
Let $\sigma$ be as in Lemma~\ref{sobolev}. Let us notice that for any $\epsilon >0$,  we have $w\in C^0([\epsilon, T]; \mathcal{H}^{\sigma})$, and consequently 
 $$ w \in C^0([\epsilon, T]\times M), \qquad \sup_{(t,x)\in (\epsilon, T)\times E} |w| (t,x) \leq C \| w_0 \|_{L^2(M)}$$
 Consider the set 
 $$X= \{ e^{(T-t)\Delta} w_0 \mid_{J\times E} \textrm{ such that } w_0 \in L^2(M)   \}. $$
 We endow $X$ with the norm inherited from  $L^1((0,T);L^\infty (M))$ and have 
 $$\| w \|_{X} =\| w \|_{L^1(J; L^\infty(  E))} \leq C \sup_{(t,x)\in (\epsilon, T)\times E} |w| (t,x)\leq C \| v_0 \|_{L^2(M)}.$$
By the observation estimate~\eqref{obsbis}, applied to $\widetilde{J} \times E, \widetilde{J} = T \setminus J$, we have 
\begin{equation}
 \label{observation-raf}
  \| w\mid_{t=0} \|_{L^2(M)} \leq C\int_{\widetilde{J} }\| e^{s\Delta} w_0\|_{L^\infty (E) }ds = C\int_{J} \| e^{(T-t)\Delta} w_0\|_{L^\infty (E) }dt\leq C\|w \|_{X}. 
  \end{equation}
As a consequence, for any $u_0,v_0 \in L^2(M)$, the map 
$$ w \in X \mapsto \bigl( w\mid_{t=0}, u_0-v_0 \bigr)_{L^2(M)} $$ is well defined because if $w_1= w_2 \in X$, then from~\eqref{observation-raf}, $w_1 \mid_{t=0} = w_2 \mid_{t=0}$. Also from~\eqref{observation-raf}, this map is a continuous linear form on $X$. By Hahn-Banach Theorem, there exists an extension as a continuous linear form to the whole space 
$$L^1((0,T); C^0 (E)). $$  
 By Riesz representation Theorem, there exists 
$$ \mu\in L^\infty ((0,T); \mathcal{M}(E))$$ (here $\mathcal{M}(E)$ is the set of Borel measures on the metric space $E$) such that 
this linear form is given~by 
$$ w\in L^1((0,T); C^0 (E)) \mapsto \int_{(0,T) \times E}  w(t,x) d\mu.$$
We can extend $\mu$ by restriction to $L^1((0,T); C^0(M))$ in the following way  
$$ w \in L^1((0, T); C^0(M)) \mapsto \int_{(0,T) \times E}  w\mid_{((0,T) \times E)}(t,x) d\mu,
$$ 
which defines an element (still denoted by $\mu$) of $L^\infty ((0,T); \mathcal{M}(M))$, supported on $[\epsilon,T]\times E$ (here we have used that $E$ is a closed set by assumption).

Let us now check that the solution to 
\begin{equation}\label{eq-mesure}
\begin{gathered}
(\partial_t - \Delta ) z = \mu (t,x)1_{E\times (0,T)}, \quad z \mid_{t=0} = 0,\\
z \mid_{\partial M} =0 \text{ (Dirichlet condition) or } \partial_\nu z \mid_{\partial M} =0 \text{ (Neumann condition)},
\end{gathered}
\end{equation}
satisfies 
$$ z \mid_{t=T} =e^{T\Delta} (u_0-v_0), z\mid_{t=0} =0$$
and consequently choosing
$$u = e^{t\Delta} u_0 -z $$ proves the second part in Theorem~\ref{control-haus}.
First we have to make sense of~\eqref{eq-mesure} (and show that the right hand side  $\mu (t,x)1_{E\times (0,T)}$ is an admissible source term). 
The first step is to prove that $\mathcal{H}^\sigma$ is dense in $C^0(M)$ subject to Dirichlet boundary conditions (or Neumann in a sense to be precised). Of course, the set $\mathcal{H}^\sigma$ being defined in terms of the eigenfunctions of the Laplace operator with Dirichlet or Neuman boundary conditions depends on this choice of boundary conditions, and in the next lemma, we make this dependence explicit.
\begin{lem}
For all $\sigma >0$, the set $\mathcal{H}^\sigma _D$ is dense in the set of continuous functions on $\overline{M}$ vanishing on $\partial M$, while the set $\mathcal{H}^\sigma _N$  is dense in the set of continuous functions on $\overline{M}$.
\end{lem}
\begin{proof} Let $u_0 \in C^0(\overline{M})$ vanishing at the boundary $\partial M$.  Then the function defined on the double manifold by 
$$\widetilde{u}_0 (x, \pm 1) = \pm  u_0 (x)$$ 
is clearly continuous on the double manifold $\widetilde{M}$. We shall say that $\widetilde{u}_0$ is odd. 
Clearly the set of odd $C^1$ functions on $\widetilde{M}$ is dense in the set of $C^0$ odd  functions on $\widetilde{M}$. Now for any $\widetilde{v}_0 $ $C^1$ and odd, working in the double manifold, we can apply the maximum principle 
for the heat semigroup $(e^{t \widetilde{\Delta}})_{t\geq 0}$, whereby the family $(e^{t\widetilde{\Delta}} \widetilde{v}_0)_{t \geq 0}$ is uniformly bounded in $L^\infty(M)$ by $\| \widetilde{v}_0 \|_{L^\infty(M)}$. Then applying again the maximum principle to $(\nabla_x e^{t\widetilde{\Delta}} \widetilde{v}_0)_{t \geq 0}$ we get that $(e^{t\widetilde{\Delta}} \widetilde{v}_0)_{t \geq 0}$ is bounded in $W^{1, \infty}(\widetilde{M})$. It clearly converges to $v_0$ in $H^1(M)= W^{1,2}(M)$ when $t \rightarrow 0$ (by decomposition on the eigenbasis of $\widetilde{\Delta}$ defined in Remark~\ref{rema.eigen}), and consequently it converges to $u_0$ in $W^{1, p}(M)$ for all $2\leq p <+\infty$, which implies convergence to $\widetilde{v}_0$ in $C^0(M)$. Now the decomposition of $\widetilde{v}_0$ on the set of eigenfunctions of the Laplace operator $\widetilde{\Delta}$ involves only odd eigenfunctions hence eigenfunctions $\widetilde{e}$ which are of the form 
$$\widetilde{e} (x, \pm 1) = e(x),$$
where $e$ is an eigenfunction of the Laplace operator on $M$ with Dirichlet boundary conditions (see Remark~\ref{rema.eigen}). As a consequence, for any $t>0$, 
$$e^{t \widetilde{\Delta}} \widetilde{v}_0 \mid_{M\times \{1\}} = e^{t\Delta_D} v_0 \in \mathcal{H}_D^\sigma (M). $$
This implies that $\mathcal{H}_D^\sigma(M)$ is dense in the set of continuous functions in $\overline{M}$ vanishing on $\partial M$.. 
To prove that $\mathcal{H}_N^\sigma(M)$ is dense in the set of continuous functions in $\overline{M}$, we proceed similarly replacing the odd extension by the even extension
$$\widetilde{u}_0) (x, \pm 1) =   u_0 (x),$$  
which sends the set of continuous functions on $\overline{M}$ to the set of continuous functions on $\widetilde{M}$ (here we do not require the vanishing of $u_0$ at the boundary). 
\end{proof}
The density of $\mathcal{H}_N^{\sigma} $ in $C^0$ implies that the map 
$$\nu \in \mathcal{M}(M) \mapsto \widetilde{\nu}= \nu \mid_{\mathcal{H}_N^\sigma} \in \mathcal{H}_N^{- \sigma},$$
 is onto and consequently any measure $\nu \in \mathcal{M}(M)$  can be seen as an element of $\mathcal{H}_N^{-\sigma}$, the dual space of $\mathcal{H}_N^\sigma$. Respectively, since $\mathcal{H}_D^{\sigma} $ is dense in the set of functions vanishing on $\partial M$, 
any measure $\nu \in \mathcal{M}(M)$ supported away from the boundary  can be seen as an element of $\mathcal{H}_D^{-\sigma}$, the dual space of $\mathcal{H}_D^\sigma$.
As a consequence, we can solve ~\eqref{eq-mesure} by using the natural spectral decomposion in $\mathcal{H}^{-\sigma}$, i.e.,
$$ \mu = \sum_k \langle \mu, e_k\rangle (t) e_k, $$
with $\langle \mu, e_k\rangle (t)$ supported in $(\epsilon, T)$ and 
$$ \text{supess}_{t \in (0, T)} \sum_k \lambda_k ^{-2\sigma} |\langle \mu, e_k\rangle|^2 (t)  <+\infty.$$
Let $w_0 \in L^2(M)$ and let $w_N$ be the solution to~\eqref{eq-obs2} with $v_0$ replaced by 
$$w_{0,N}= \sum_{k\leq N} (w_0, e_k) e_k,$$
and $z_N$ the solution to~\eqref{eq-mesure}, where $\mu$ is replaced by 
$$\mu_N= \sum_{k\leq N}\langle \mu, e_k\rangle (t) e_k.$$
We have 
\begin{multline}\label{ipp}
0= \int_0^{T} \bigl( (\partial_t + \Delta ) w_N, z_N\bigr)_{L^2} =  \Bigl[\bigl(  w_N, z_N\bigr)_{L^2}\Bigr]_0^{T} - \int_0^{T} \bigl(  w_N, (-\partial_t + \Delta )z_N\bigr)_{L^2} \\
=  \bigl(  w_{0,N}, z_N\mid_{t=T}\bigr)_{L^2} - \int_0^{T} \bigl(  w_N, \mu_N\bigr)_{L^2}.
\end{multline}
We now let $N$ tend to infinity. Then 
\begin{equation}
\begin{gathered}
w_{0,N} \rightarrow w_0 \text{ in } L^2,\\
z_N \mid_{t=T} \rightarrow z \mid_{t=T} \text{ in }\mathcal{H}^{-\sigma} \Rightarrow z_N \mid_{t=T} \rightarrow z \mid_{t=T'} \text{ in }L^2,\\
w_N \rightarrow w \text{ in } C^0([0, T]; \mathcal{H}^\sigma),\\
\mu_N \rightarrow \mu \text{ in } L^\infty([0, T]; \mathcal{H}^{-\sigma}).
\end{gathered}
\end{equation}
We deduce that we can pass to the limit in~\eqref{ipp} and get 
$$0  =  \bigl(  w_{0}, z\mid_{t=T}\bigr)_{L^2} - \int_0^{T} w(t,x) 1_{t\in (0,T)}d \mu.$$
From the definition of $\mu$ we have 
$$ \int_0^{T} w(t,x) d \mu= \bigl( w\mid_{t=0}, u_0-v_0\bigr)_{L^2}=\bigl( e^{T\Delta} w_0, u_0-v_0\bigr)_{L^2}.$$
We finally get
$$ \forall w_0 \in L^2, \bigl(  w_{0}, z\mid_{t=T}\bigr)_{L^2}= \bigl( e^{T\Delta} w_0, u_0-v_0\bigr)_{L^2} \Rightarrow z\mid_{t=T}= e^{T \Delta}(u_0- v_0).$$
$u= e^{t\Delta} u_0- z$ satisfies the  second part of Theorem~\ref{control-haus}. 

We now turn to the second part in Theorem~\ref{control-hausbis} and highlight the modifications required in the proof above. We shall focus on the case $E= E_2$ and assume that $E$ satisfies~\eqref{eq.disjoint}. Let $J= \{ t_n, n\in \N\} \cup\{T\} \subset [t_0, T]$ (recall that $t_0 >0$), $\widetilde{J} = T-J= \{ s_n\} \cup \{0\}$. Let
$$X= \{ e^{(T-t)\Delta} v_0 \mid_{J\times E}, \textrm{ with } v_0 \in L^2(M)  \} \subset C^0(J \times E),$$ endowed with the $\sup$ norm.  Then  according to ~\eqref{obsterbis}, with $s_n = T- t_n$, the linear form
$$w\in X\mapsto \bigl( w\mid_{t=0}, u_0-v_0 \bigr)_{L^2(M)}$$
is well defined and continuous, and more precisely bounded by 
\begin{equation} \label{bornebis}C \sup_{n\in \N, x\in E} e^{-\frac{B}{ T- t_n}}|w(t_n,x)|.
\end{equation}
Indeed (notice that $T- t_n \geq t_{n+1} - t_n$),
$$\| w\mid_{t=0} \|_{L^2(M)}=   \| e^{T\Delta } v_0 \|_{L^2(M)}\leq C \sup_{n\in \N, x\in E} e^{-\frac{B}{ t_{n+1}- t_n}}|w(t_n,x)|\leq C \sup_{n\in \N, x\in E} e^{-\frac{B}{ T- t_n}}|w(t_n,x)|.$$
According to Hahn Banach theorem~\cite[Theorem 3.2]{Ru91}, we can extend this map to the whole space $C^0(J \times E),$ so that it is still bounded by~\eqref{bornebis}. By Riesz representation Theorem, this continuous linear form can be represented as a measure $\mu\in \mathcal{M}(J\times E)$ which still satisfies the same bound~\eqref{bornebis}. As previously we can extend this measure as a measure on $[0,T]\times M$ which is supported in $J\times E$. Hence this measure takes the form 
$$\mu = \sum_{n}\delta_{t= t_n} \otimes \mu_n + \delta_{t=T}\otimes \mu_\infty,$$ with $\mu_j, \mu_\infty$ measures on $M$ supported by $E$. Using~\eqref{bornebis} we get that 
$$ \sum_{n} e^{\frac{B}{T- t_n} }  |\mu_n| (E) <+\infty, \qquad\mu_{\infty} =0.$$
Now we can simply make sense of solving
$$(\partial_t - \Delta) z= \sum_{n} \delta_{t= t_n} \otimes \mu_n, z\mid_{ty=0} = 0,$$
with Dirichlet or Neumann boundary conditions
in $L^\infty([0,T); \mathcal{H}^{-\sigma})$, by simply noticing that the solution to this equation is the solution to the homogeneous heat equation on $(t_n, t_{n+1})$ which satisfies the jump condition
$$ z\mid_{t_n+0} - z\mid_{t_n -0} = \mu_n \in \mathcal{H}^{-\sigma}.$$
Since
 $$\sum_n \| \mu_n\|_{\mathcal{H}^{- \sigma}}\leq C\sum_n |\mu_n|(E)<+\infty,$$
we deduce that actually $\lim_{t\genfrac{}{}{0pt}{}{<}{\rightarrow}T} z(t)$ exists in $\mathcal{H}^{- \sigma}$, and consequently the solution exists and is unique in $[0, +\infty)$ (defined on $[T,+\infty)$ as the solution  of the homogeneous heat equation). We now write the analog of the integration by parts formula~\eqref{ipp}.
Let $z_N, w_N$ and $\mu_{n,N}$ be the projections of$z,v$ and $\mu_n$ on the space spanned by the $N$ first eigenfunctions. On $(t_n, t_{n+1})$, we have 
\begin{multline}\label{ippbis}
0= \int_{t_n}^{t_{n+1}} \bigl( (\partial_t + \Delta ) w_N, z_N\bigr)_{L^2} =  \Bigl[\bigl(  w_N, z_N\bigr)_{L^2}\Bigr]_{t_n}^{t_{n+1}} - \int_{t_n}^{t_{n+1}} \bigl(  w_N, (-\partial_t + \Delta )z_N\bigr)_{L^2} \\
=  \bigl(  w_{N}\mid_{t_{n+1}}, z_N\mid_{t=t_{n+1}-0}\bigr)_{L^2} - \bigl(  w_{N}\mid_{t_{n}}, z_N\mid_{t=t_{n}+0}\bigr)_{L^2}
\end{multline}
which implies (using that $z_N\mid_{t=0}=0$ and $\lim_{n \rightarrow +\infty} w_N\mid_{t=t_n} = w_N (T))$ the following
\begin{multline}\label{ippter}
0= \int_{0}^{T} \bigl( (\partial_t + \Delta ) w_N, z_N\bigr)_{L^2} =  \sum_{n}  \bigl(  w_{N}\mid_{t_{n+1}}, z_N\mid_{t=t_{n+1}-0}\bigr)_{L^2} - \bigl(  w_{N}\mid_{t_{n}}, z_N\mid_{t=t_{n}+0}\bigr)_{L^2}\\
= \lim_{k\rightarrow + \infty}  \bigl(  w_{N}(t_{k}), z_N\mid_{t=t_k-0}\bigr)_{L^2}+ \sum_n^{k-1} \bigl(  w_{0,N}\mid_{t_{n}}, z_N\mid_{t=t_{n}-0}-z_N\mid_{t=t_{n}+0}\bigr)_{L^2}\\
=  \bigl(  w_{N}(T), z_N\mid_{t=T}\bigr)_{L^2} - \sum_n  \bigl(  w_{0,N}\mid_{t_{n}}, \mu_{n,N}\bigr)_{L^2}\\
=  \bigl(  w_{0,N}, z_N\mid_{t=T}\bigr)_{L^2}-\int_0^{T}  w_N(t) d\mu_{N}\
\end{multline}
We can now pass to the limit $N \rightarrow + \infty$ and get
$$\bigl(  u_0 -v_{0}, z\mid_{t=T}\bigr)_{L^2}=-  \int_0^{T}  w(t) d\mu,$$
and we conclude as previously that $u= e^{t\Delta}u_0 -z$ satisfies (with Dirichlet or Neumann boundary conditions)
$$(\partial_t- \Delta) u = - \sum_{n} \delta_{t= t_n} \otimes \mu_n, \qquad u\mid_{t=0} = u_0, u\mid_{t=T} =v_0.
$$
This proves the second part in Theorem~\ref{control-hausbis}, in the case $E= E_2$.The case $E=E_1$ is proved similarly by replacing in the proof above~\eqref{obsterbis} by \eqref{obster}.

\def\cprime{$'$} \def\cprime{$'$}

\end{document}